\documentclass[11pt]{amsart}
\usepackage{amssymb, amsmath, amsthm}
\usepackage{amscd,fancybox}
\usepackage{fullpage}
\input amssym.def
\input amssym

\newtheorem{theorem}{Theorem}[section]

\newtheorem{prop}[theorem]{Proposition}
\newtheorem{cor}[theorem]{Corollary}
\newtheorem{lem}[theorem]{Lemma}
\newtheorem{remark}[theorem]{Remark}
\newtheorem{defn}[theorem]{Definition}

\newcommand{\Lnu}[0]{\hat{L}_\nu}

\def \Z{\mathbb{Z}}

\def \Q{\mathbb{Q}}
\def \C{\mathbb{C}}

\def \h{\frak{h}}

\begin{document}

\title[Vertex-algebraic structure of principal subspaces]
{Vertex-algebraic structure of principal subspaces of standard
  $A^{(2)}_2$--modules, I}

\author{C. Calinescu, J. Lepowsky and A. Milas}

\begin{abstract}
  Extending earlier work of the authors, this is the first in a series
  of papers devoted to the vertex-algebraic structure of principal
  subspaces of standard modules for twisted affine Kac-Moody
  algebras. In this part, we develop the necessary theory of principal
  subspaces for the affine Lie algebra $A_2^{(2)}$, which we expect
  can be extended to higher rank algebras.  As a ``test case,'' we
  consider the principal subspace of the basic $A_2^{(2)}$-module and
  explore its structure in depth.

\end{abstract}

\maketitle

\section{Introduction}

The theory of principal subspaces of standard modules for untwisted
affine Kac-Moody algebras was initiated in influential work of Feigin
and Stoyanovsky (\cite{FS1}, \cite{FS2}) and has been further
developed by several authors from different standpoints.  Our approach
to principal subspaces (see \cite{CalLM1}, \cite{CalLM2},
\cite{CalLM3}, \cite{MPe} and the earlier work \cite{CLM1},
\cite{CLM2}) is based on vertex (operator) algebra theory
(cf. \cite{B}, \cite{FLM2}, \cite{FHL}, \cite{LL}) in such a way that
the principal subspaces of standard modules of a fixed level are
viewed as modules for the principal subspace of the ``vacuum'' module,
itself viewed as a vertex algebra.  In these works, we have proved
that all the important algebraic and combinatorial properties of
principal subspaces (e.g., multi-graded dimensions), which are a
priori unknown, can be extracted from standard modules by using
certain natural vertex-algebraic concepts such as intertwining
operators among modules (although it is fair to say that the full
power of intertwining operators is yet to be explored), annihilating
ideals, etc. As is customary in this area of representation theory, we
often use vertex operator constructions of basic modules and then
tensor products to study the (more difficult) higher level standard
modules (cf. \cite{CalLM2} and \cite{CLM2}).  We also point out that
there are other types of (commutative) ``principal subspaces,'' which
can be also studied by using our techniques (see \cite{Pr}, \cite{Je},
etc.).

In this series of papers we switch our attention to standard modules
for {\it twisted} affine Lie algebras and their principal
subspaces. As we shall see, in the twisted case new phenomena arise
and the theory is much more subtle than in the untwisted case. This is
the main reason why in this paper we focus primarily on the simplest
twisted affine Lie algebra $A_2^{(2)}$, which already illustrates the
difficulty of the general case. In parallel with our first paper on
the untwisted $A_1^{(1)}$ case \cite{CalLM1}, we first study the level
one standard $A_2^{(2)}$-module $V_L^T$ and its principal subspace
$W_L^T$ and carry out the analysis in full detail. In particular, we
are able to obtain a description of $W_L^T$ via generators and
relations. Then, by using certain natural maps we build a short exact
sequence, which gives rise to a simple recursion for the graded
dimension of $W_L^T$.  From this recursion we easily infer that this
graded dimension is given by the number of integer partitions into
distinct odd parts.  We stress that the main goal of this paper is not to
obtain this classical recursion but rather to set up a new twisted
vertex-algebraic theory of principal subspaces that
in the first test case produces a structure that naturally underlies
this recursion; in \cite{CalLM1} (and \cite{CLM1}) it was
the classical Rogers-Ramanujan recursion that we obtained from the
corresponding (untwisted) vertex-algebraic
structure.  Interestingly, intertwining operators do not play a
leading role in the present paper; instead, certain twisted operators,
similar to simple current operators used in \cite{DLM}, have turned
out to be more essential.  This is related to the fact that there is
only one level one standard $A_2^{(2)}$-module, while for $A_1^{(1)}$
there are two, and in \cite{CalLM1} (and \cite{CLM1}), an intertwining
operator relating them led to the basic structure.  These twisted
operators act naturally on the principal subspace, as certain shifting
operators.  In a sequel to this paper, we shall extend this work to
higher-level standard $A_2^{(2)}$-modules.

Let us briefly elaborate on the contents of the paper.  In Section 2,
starting from an arbitrary isometry of an arbitrary positive-definite
even lattice $L$, we review the relevant parts of the theory of
twisted modules for lattice vertex operator algebras $V_L$. This
material is mostly taken from \cite{FLM1}, \cite{FLM2}, \cite{FLM3}, \cite{L1}
and \cite{L2}.  (See also \cite{DL1} and \cite{BHL}.)  We present it in this
generality as the foundation of our sequels to the present paper; but
in addition, the $A_2^{(2)}$ case, the main focus of this paper, in
fact exhibits the main subtleties of the general theory. Section 3 is
devoted to recalling the vertex-algebraic construction of the twisted
affine Lie algebra $A_2^{(2)}$ and of its basic module $V_L^T$;
starting here, $L$ is the root lattice of type $A_2$. Section 4 deals
with shifted modules and operators; they will become more relevant in
our future publications. In Section 5 we introduce and investigate the
principal subspace $W_L^T \subset V_L^T$. Section 6
contains the definitions and properties of various maps
needed in Section 7 to prove a presentation of the principal subspace
$W_L^T$ (see Theorem \ref{main-thm}).  Finally, Section 8 deals with a
reformulation of one of the main results in \cite{CalLM3}, but this
time, in the spirit of the present paper, without the use of
intertwining operators.

\section{Lattice vertex operator algebras and twisted modules}

In this section we recall, for the reader's convenience, the vertex
operator constructions associated to a general positive-definite even
lattice equipped with a general isometry, which is necessarily of
finite order.  We review the relevant results of \cite{L1}, \cite{FLM2},
\cite{FLM3} and \cite{L2}.  We use the notation and
terminology of \cite{FLM3} (in particular, Chapters 7 and 8) and
\cite{LL} (in particular, Sections 6.4 and 6.5). In fact, these
general constructions of lattice vertex operator algebras and of
twisted modules have been presented in \cite{BHL} (for the
purpose of giving an equivalence of two constructions of
permutation-twisted modules), and the descriptions of the
constructions and results in this section are similar to the
corresponding review in \cite{BHL} of these earlier results.  In the
next section we will specialize the general setting to the root
lattice of $\frak{sl}(3, \C)$ and to a certain isometry of this root
lattice.

We work in the following setting, under the assumptions made in
Section 2 of \cite{L1}: Let $L$ be a positive-definite even lattice
equipped with a (nondegenerate symmetric) $\mathbb{Z}$-bilinear form
$\langle \cdot, \cdot \rangle$, and let $\nu$ be an isometry of $L$
and $k$ a positive integer such that
\begin{equation}
\nu^k = 1.
\end{equation}
Note that $k$ need not be the exact order of $\nu$ and,
in fact, the appropriate period $k$ of $\nu$
will (necessarily) be larger than the order of $\nu$ in our
specialized setting (see the next section).  We also assume that if
$k$ is even, then
\begin{equation}\label{even-equation}
\langle \nu^{k/2} \alpha, \alpha \rangle \in 2 \mathbb{Z} 
\ \ \ \ \mbox{for $\alpha \in L$},
\end{equation}
which can always be arranged by doubling $k$ if necessary. Under these
assumptions we have
\begin{equation}\label{even-condition}
\left< \sum_{j = 0}^{k-1} \nu^j \alpha, \alpha \right> \in 2 \mathbb{Z}
\end{equation}
for $\alpha \in L$. This doubling procedure will be relevant in the
next section, where $\nu^2=1$ but $k$ cannot be 2; it will be 4.

We continue to quote from \cite{L1} and \cite{FLM2}. Let $\eta$ be a
fixed primitive $k^{\mbox{\rm th}}$ root of unity.  The functions
$C_0$ and $C$ defined by
\begin{eqnarray}\label{commutator-definition-0}
C_0: L \times L &\longrightarrow& \C^\times \\ (\alpha, \beta)
&\mapsto& (-1)^{ \langle \alpha, \beta \rangle} \nonumber
\end{eqnarray}
and
\begin{eqnarray}\label{commutator-definition}
C: L \times L &\longrightarrow& \C^\times \\ (\alpha, \beta) &\mapsto&
(-1)^{\sum_{j = 0}^{k-1} \langle \nu^j \alpha, \beta \rangle}
\eta^{\sum_{j = 0}^{k-1} \langle j \nu^j \alpha , \beta \rangle}
\nonumber \\ & &\quad = \prod_{j=0}^{k-1} (-\eta^{j} )^{\langle \nu^j
\alpha, \beta \rangle} \nonumber
\end{eqnarray}
are bilinear into the abelian group $\C^\times$ and are
$\nu$-invariant.  Clearly,
\begin{equation} \label{c-0=1}
C_0 (\alpha, \alpha) =1.
\end{equation}
Also,
\begin{equation} \label{c=1}
C(\alpha, \alpha)=1,
\end{equation}
whose proof uses (\ref{even-condition}) and thus our assumption
(\ref{even-equation}).  

Set
\begin{equation} \label{eta-0}
\eta_0 = (-1)^k \eta.
\end{equation}

Since $C_0$ and $C$ are alternating bilinear maps into
$\mathbb{C}^{\times}$ (that is, they satisfy (\ref{c-0=1}) and
(\ref{c=1})), they determine uniquely (up to equivalence) two
central extensions
\begin{equation}\label{exact-0}
1 \rightarrow \langle \eta_0 \rangle \rightarrow \hat{L}
\bar{\longrightarrow} L \rightarrow 1
\end{equation}
and
\begin{equation}\label{exact}
1 \rightarrow \langle \eta_0 \rangle \rightarrow \Lnu
\bar{\longrightarrow} L \rightarrow 1
\end{equation}
of $L$ by the cyclic group $\langle \eta_0 \rangle$ with commutator
maps $C_0$ and $C$, respectively:
\begin{equation} \label{commutator=C0}
aba^{-1} b^{-1} = C_0( \overline{a}, \overline{b}) 
\qquad \mathrm{for} \quad
a,b \in \hat{L}
\end{equation}
and
\begin{equation} \label{commutator=C}
aba^{-1} b^{-1} = C( \overline{a}, \overline{b}) \qquad
\mathrm{for} \quad a,b \in \Lnu .
\end{equation}
There is a natural set-theoretic identification (which is not an
isomorphism of groups unless $k = 1$ or $k = 2$) between the groups
$\hat{L}$ and $\Lnu$ such that the respective group multiplications,
denoted here by $\times$ and $\times_\nu$ to distinguish them, are
related as follows:
\begin{equation}\label{identify-central-extensions}
a \times b = \prod_{-k/2<j<0} (- \eta^{-j})^{\langle \nu^{j} \bar{a}, \bar{b}
\rangle} a \times_\nu b \qquad \mathrm{for} \quad a,b \in \hat{L}.
\end{equation}

As in \cite{L1}, let 
\begin{eqnarray} \label{section}
e: L & \longrightarrow & \hat{L} \\ \nonumber
\alpha & \mapsto & e_{\alpha} \nonumber
\end{eqnarray}
be a normalized section of $\hat{L}$, that is, 
$$
e_0=1
$$
and
$$
\overline{e_{\alpha}}=\alpha \; \; \mbox{for all} \; \;  \alpha \in L.
$$ 
Similarly, we have $e: L \longrightarrow \hat{L}_{\nu}$, $\alpha
\mapsto e_{\alpha}$ a normalized section of $\hat{L}_{\nu}$. The
function
\begin{equation} \label{epsilon_c} 
\epsilon_C:L \times L \longrightarrow \langle \eta_0 \rangle 
\end{equation}
defined by
\begin{equation} \label{e-L-nu}
e_{\alpha}e_{\beta}=\epsilon_C(\alpha, \beta) e_{\alpha+\beta} \; \;
\mbox{for} \; \; \alpha, \beta \in L
\end{equation}
is a normalized $2$-cocycle associated with the commutator map $C$,
that is,
\begin{equation} \label{epsilon-1}
\epsilon_C (\alpha, \beta)\epsilon_C(\alpha+\beta,
\gamma)=\epsilon_C(\beta, \gamma) \epsilon_C(\alpha, \beta+\gamma),
\end{equation}
\begin{equation} \label{epsilon-2}
\epsilon_C(0, 0)=1,
\end{equation}
\begin{equation} \label{epsilon-3}
\epsilon_C(\alpha, \beta)/\epsilon_C(\beta, \alpha)=C(\alpha, \beta).
\end{equation}
Also, the function 
\begin{equation} \label{epsilon_c_0}
\epsilon_{C_0}:L \times L \longrightarrow \langle \eta_0 \rangle
\end{equation}
defined by
\begin{equation} \label{epsilon_c_0-def}
\epsilon_{C_0}(\alpha, \beta)= \prod_{-k/2 < j < 0}(-\eta^{-j}
)^{\langle \nu^j \alpha, \beta \rangle} \epsilon_{C}(\alpha, \beta)
\end{equation}
is a normalized $2$-cocycle associated with the commutator map
$C_0$. Then
\begin{equation} \label{prod-C_0}
e_{\alpha}e_{\beta} =\epsilon_{C_0}(\alpha, \beta) e_{\alpha+\beta} \; \; \;
\mbox{in} \; \; \; \hat{L}
\end{equation}
(by (\ref{identify-central-extensions}), (\ref{e-L-nu}) and
(\ref{epsilon_c_0-def}))
and 
\begin{equation} \label{quotient-C_0}
\epsilon_{C_0}(\alpha, \beta)/\epsilon_{C_0}(\beta, \alpha)=
C_0(\alpha, \beta)
\end{equation}
(recall formula (4.6) in \cite{L1}).

There exists an automorphism $\hat{\nu}$ of $\hat{L}$ (fixing
$\eta_0$) such that
\begin{equation}\label{nu-on-hatL-0}
\overline{\hat{\nu} a} = \nu \overline{a} \quad \mathrm{for} \quad a \in
\hat{L},
\end{equation}
that is, $\hat{\nu}$ is a lifting of $\nu$. The map $\hat{\nu}$ is
also an automorphism of $\Lnu$ satisfying
\begin{equation}\label{nu-on-hatL}
\overline{\hat{\nu} a} = \nu \overline{a} \quad \mathrm{for} \quad a
\in \Lnu.
\end{equation}
We may and do choose $\hat{\nu}$ so that
\begin{equation}\label{nuhata=a}
\hat{\nu} a = a \quad \mathrm{if} \quad 
\nu \overline{a} = \overline{a}.
\end{equation}
Then
\begin{equation}\label{nuhat^k=1}
\hat{\nu}^k = 1.
\end{equation}
See \cite{L1} for these nontrivial facts.

Following the treatments in \cite{L1} and \cite{FLM2} (see also
\cite{FLM3} and \cite{DL1}) we will construct a vertex operator
algebra $V_L$ equipped with an automorphism $\hat{\nu}$, using the
central extension $\hat{L}$. We will then use the central extension
$\Lnu$ to construct $\hat{\nu}$-twisted modules for $V_L$.

Embed $L$ canonically in the $\C$-vector space $$\h = \C
\otimes_\mathbb{Z} L$$ and extend the $\Z$-bilinear form on $L$ to a
$\C$-bilinear form $\langle \cdot, \cdot \rangle$ on $\h$.  Viewing
$\frak{h}$ as an abelian Lie algebra, we have the corresponding affine Lie
algebra
\begin{equation} \label{hat-h}
\hat{\h} = \h \otimes \C[t,t^{-1}] \oplus \C {\bf k}
\end{equation}
with the brackets 
\begin{equation}
[\alpha \otimes t^m, \beta \otimes t^n]=\langle \alpha , \beta \rangle
m\delta_{m+n,0}{\bf k}\ \ {\rm for}\ \ \alpha, \beta \in{\h}, \ \
m,n\in \Z
\end{equation}
and
\begin{equation}
[{\bf k},\hat{\h}]=0.
\end{equation}
There is a $\Z$-grading on $\hat{\h}$, called the {\it weight
grading}, given by
$$
{\rm wt}\,(\alpha \otimes t^m)=-m \ \ {\rm and} \ \ {\rm wt}\, {\bf
k}=0
$$
for $\alpha \in {\h}$ and $m\in \Z$.  Consider the following
subalgebras of $\hat{\h}$:
$$
\hat{\h}^+={\h}\otimes t \C[t] \ \ {\rm and} \ \
\hat{\h}^-={\h}\otimes t^{-1} \C[t^{-1}]
$$
and
$$
\hat{\h}_{\Z} = \hat{\h}^+\oplus\hat{\h}^-\oplus \C {\bf k}.
$$
The latter is a Heisenberg algebra, in the sense that its
commutator subalgebra equals its center, which is one-dimensional.
Form the induced $\hat{\h}$-module
\begin{equation} \label{induced-h}
M(1)=U(\hat{\h})\otimes_{U( \h \otimes \C [t] \oplus \C {\bf k})} \C
\simeq S(\hat{\h}^-)\ \ \ (\mbox{linearly}),
\end{equation}
where ${\h}\otimes \C [t]$ acts trivially on $\C$ and ${\bf k}$ acts
as 1.  This is an ireducible $\hat{\h}_{\Z}$-module, $\Z$-graded by
weights:
$$
M(1)=\coprod_{n \geq 0} M(1)_n,
$$
where $M(1)_n$ denotes the homogeneous subspace of weight $n$. Note
that $\mbox{wt}\,1\ =\ 0$ (by $1$ we mean $1\otimes 1$).

Now form the induced $\hat{L}$-module
\begin{equation}
\C \{L\} = \C [\hat{L}] \otimes_{\mathbb{C}[ \langle \eta_0 \rangle]
} \C \simeq \C[L] \; \; \; (\mbox{linearly}). 
\end{equation}
For $a \in \hat{L}$, write 
$$
\iota(a)=a \otimes 1 \in \mathbb{C} \{ L\}.
$$ 
The space $\C \{L\}$ is $\C$-graded:
\begin{equation} \label{weight-a}
\mbox{wt}\,\iota(a)=\frac{1}{2}\langle \overline{a},
\overline{a}\rangle \ \ \ \
\mbox{for}\ \
a\in \hat{L}.
\end{equation}
In particular, we have 
\begin{equation} \label{weight-1}
\mbox{wt} \,\iota(1)=0
\end{equation}
and
\begin{equation} \label{weight-e}
\mbox{wt} \: \iota(e_{\alpha})= \frac{1}{2} \langle \alpha, \alpha 
\rangle.
\end{equation}
The action of $\hat{L}$ on
$\C\{L\}$ is given by
\begin{equation}
a\cdot\iota(b)=\iota(ab)
\end{equation}
for $a,b\in\hat{L}$.  There is also a grading-preserving action of
${\h}$ on $\C\{L\}$ given by
\begin{equation}
h\cdot\iota(a)= \langle h,\overline{a}\rangle \iota(a)
\end{equation}
for $h\in{\h}$. Define the operator $x^{h}$ as follows: 
\begin{equation}
x^h\cdot\iota(a) = x^{\langle h,\overline{a}\rangle }\iota(a)
\end{equation}
for $h\in{\h}.$

Set
\begin{eqnarray} \label{voa}
V_L &=& M(1)\otimes_{\C} \C \{L\}\\ &\simeq & S(\hat{\h}^-)\otimes \C
[L] \ \ \ \ \qquad (\mbox{linearly}) \nonumber
\end{eqnarray}
and set ${\bf 1}=1 \otimes \iota(1)$. Give $V_L$ the tensor product
weight grading:
$$
V_L=\coprod_{n\in \C}(V_L)_n.
$$
We have that $\hat{L}$, $\hat{\h}_{\Z}$, $\h$, $x^h$ $(h\in{\h})$ act
naturally on $V_L$ by acting on either $M(1)$ or $\C \{L\}$ as
indicated above. In particular, ${\bf k}$ acts as 1.

Throughout this paper, we shall use the symbol $x$ as a formal
variable, and symbols such as $x_i$ (along with $x$), $y_i$, $q$ and
$t$ (already used above) will be used for independent commuting formal
variables.  (Note that $x$, etc., will sometimes refer to other
things, in context.)

For $\alpha \in \h$, $n\in \Z$, we write $\alpha(n)$ for the operator
on $V_L$ associated with $\alpha\otimes t^n \in \hat{\h}$:
$$
\alpha \otimes t^n \mapsto \alpha(n)
$$
and we set
\begin{equation}
\alpha(x)=\sum_{n\in \Z} \alpha(n) x^{-n-1}.
\end{equation}
Using the formal exponential series $\mbox{exp}(\cdot)$ define 
\begin{equation} \label{exp} 
E^{\pm}(\alpha, x)= \mbox{exp} \left (\sum_{n
 \in \pm \mathbb{Z}_{+}} \frac{\alpha(n)}{n} x^{-n} \right ) \in
 (\mbox{End} \, V_L) [[x, x^{-1}]]
\end{equation}
for any $\alpha \in \h$. As in \cite{FLM2}, for $a \in \hat{L}$, set
\begin{equation} \label{normal-ordering} 
 Y(\iota(a),x)= \ _\circ^\circ
  \ e^{\sum_{n \neq 0} \frac{-\overline{a}(n)}{n} x^{-n}}\
  _\circ^\circ a x^{\overline{a}},
\end{equation}
where by $\ _\circ^\circ \ \cdot \ _\circ^\circ$ we mean a normal
ordering procedure in which the operators $\bar{a}(n)$ ($n <0$) and
$a \in \hat{L}$ are placed to the left of the operators $ \bar{a}(n)$
($n \geq 0$) and $x^{\bar{a}}$ before the expression is evaluated. By
(\ref{exp}) the vertex operator (\ref{normal-ordering}) becomes
\begin{equation}
Y(\iota(a),x)= E^{-}(-\bar{a}, x)E^{+} (-\bar{a}, x) a x^{\bar{a}}.
\end{equation}
In particular, using the section $e$ we have the operator
\begin{equation}
Y(\iota(e_{\alpha}), x)= E^{-}(-\alpha, x)E^{+}(-\alpha, x) 
e_{\alpha} x^{\alpha}
\end{equation}
for $\alpha \in L$. For $\alpha \in L$ and $n \in \mathbb{Z}$ we
define the operators $x_{\alpha}(n)$ by the expansion
\begin{equation}
Y(\iota(e_{\alpha}), x)= \sum_{n \in \mathbb{Z}} 
x_{\alpha}(n)x^{-n-\frac{\langle \alpha, \alpha \rangle}{2}}.
\end{equation}
More generally, for an element $ v = \alpha_1(-n_1)\cdots
\alpha_m(-n_m) \otimes \iota(a) \in V_L\nonumber $ with $\alpha_1,
\dots, \alpha_m \in \h$, $n_1,\dots, n_m >0$ and $a \in \hat{L}$, we set
\begin{equation} \label{untwisted-operator}
Y(v,x) = \ _\circ^\circ \left(\frac{1}{(n_1-1)!}\left(\frac{
d}{dx}\right)^{n_1-1} \alpha_1(x)\right)\cdots
\left(\frac{1}{(n_m-1)!}
\left(\frac{d}{dx}\right)^{n_m-1}\alpha_m(x)\right)Y(\iota(a),x) \
_\circ^\circ 
\end{equation}
and this gives a well-defined linear map
\begin{eqnarray} \nonumber
V_L& \rightarrow &(\mbox{End}\,V_L)[[x, x^{-1}]] \\ \nonumber
v &\mapsto& Y(v,x)=\displaystyle{\sum_{n\in \Z}}v_nx^{-n-1}, \ \ \
v_n\in\mbox{End}\,V_L. \nonumber
\end{eqnarray}
Set 
\begin{equation} \label{omega}
\omega = \frac{1}{2} \sum_{i=1}^{{\rm dim} \, {\h}} h_i(-1)h_i(-1) {\bf 1},
\end{equation}
where $\{h_i\}$ is an orthonormal basis of $\h$. 

By Chapter 8 of \cite{FLM3}, $V_L = (V_L, Y, {\bf 1},
\omega)$ is a simple vertex operator algebra of central charge equal
to ${\rm rank} \; L ={\rm dim} \, {\h}$ (cf. Theorem 6.5.3 in
\cite{LL}). We also have that $V_L$ is independent, up to an
isomorphism of vertex operator algebras preserving the
$\hat{\h}$-module structure, of the central extension (\ref{exact-0})
subject to (\ref{commutator=C0}) and on the choices of $k >0$ and the
primitive root $\eta$ (cf. Proposition 6.5.5 and also Remarks 6.5.4
and 6.5.6 in \cite{LL}).

We will now extend the automorphism $\hat{\nu}$ of $\hat{L}$ in a
natural way to an automorphism of period $k$, also denoted by
$\hat{\nu}$, of $V_L$: The automorphism $\nu$ of $L$ acts in a natural
way on ${\h},$ on $\hat{\h}$ and on $M(1)$, preserving the
gradings. We have
\begin{equation} \label{extension-one}
\nu (u\cdot m)=\nu (u)\cdot \nu(m)
\end{equation}
for $u\in \hat{\h}$ and $m\in M(1)$. The automorphism $\hat{\nu}$ of
$\hat{L}$, extended naturally to $\C \{L\}$, satisfies the conditions
\begin{equation}
\hat{\nu}(h\cdot \iota(a))=\nu(h)\cdot \hat{\nu}\iota(a),
\end{equation}
\begin{equation}
\hat{\nu}(x^h \cdot \iota(a)) = x^{\nu(h)} \cdot \hat{\nu}\iota(a)
\end{equation}
and
\begin{equation}
\hat{\nu} (a\cdot \iota(b)) = \hat{\nu}(a) \cdot
\hat{\nu} \iota(b)
\end{equation}
for $h \in \h$ and $a, b \in \hat{L}$.  We take our extension
$\hat{\nu}$ on $V_L$ to be $\nu\otimes\hat{\nu}$. It preserves the
grading and we have
\begin{eqnarray} 
\hat{\nu}(a\cdot v) &=& \hat{\nu}(a)\cdot\hat{\nu}(v)\\
\hat{\nu}(u\cdot v) &=& \nu(u)\cdot \hat{\nu}(v) \\
\hat{\nu}(x^h\cdot v) &=& x^{\nu(h)}\cdot \hat{\nu}(v) 
\label{extension-final}
\end{eqnarray}
for $a\in \hat{L}$, $u\in \hat{\h}$, $h\in {\h}$ and $v\in V_L$. It
follows that $\hat{\nu}$ is an automorphism of the vertex operator
algebra $V_L$.

In the remainder of this section we review the construction of the
$\hat{\nu}$-twisted modules for $V_{L}$ following \cite{L1} and
\cite{FLM2} (see also \cite{DL1}). Using the primitive $k$-th root of
unity $\eta$, where $k$ is our choice of period of the isometry $\nu$,
and the vector space $\frak{h}$, set
\begin{equation}\label{hngrading}
\h_{(n)} = \{ h \in \h \; | \; \nu h = \eta^{n} h \} \subset \h
\end{equation}
for $n \in \mathbb{Z}$, so that 
$$\h = \coprod_{p \in \mathbb{Z}/k\mathbb{Z}} \h_{(p)}.
$$ Here we identify $\h_{(n \; \mathrm{mod} \; k)}$ 
with $\h_{(n)}$ for $n \in \mathbb{Z}$. 
For $p \in
\mathbb{Z}/k\mathbb{Z}$ consider the $p^{th}$ projection
\begin{equation}\label{Pp}
P_p : \h \longrightarrow  \h_{(p)} 
\end{equation}
and for $h \in \h$ and $n \in \mathbb{Z}$ set $h_{(n)}=P_{n \;
\mathrm{mod} \; k} h$.

Consider the $\nu$-twisted affine Lie algebra associated with $\h$
(viewed as an abelian Lie algebra) and $\langle \cdot, \cdot \rangle$:
\begin{equation}
\hat{\h}[\nu] = \coprod_{n\in\frac{1}{k} \Z} \h_{(kn)}\otimes
t^{n}\oplus \C {\bf k}
\end{equation}
with 
\begin{equation}
[\alpha \otimes t^m, \beta \otimes t^n]=\langle \alpha , \beta \rangle
m\delta_{m+n,0}{\bf k}
\end{equation}
for $\alpha \in{\h}_{(km)}$, $\beta \in{\h}_{(kn)}$, and
$m,n\in\frac{1}{k} \Z$, and
\begin{equation}
[{\bf k},\hat{\h}[\nu]]=0.
\end{equation}
This algebra is $\frac{1}{k} \mathbb{Z}$-graded by weights:
\begin{equation}\label{grade-h}
{\rm wt}\,(\alpha \otimes t^{m})=-m, \ \ \ {\rm wt}\,{\bf k}=0
\end{equation}
for $m \in \frac{1}{k} \Z$, $\alpha \in {\h}_{(km)}$.  Notice that for
$\nu$ the identity automorphism, $\hat{\h}[\nu]$ is the same as the
untwisted affine Lie algebra $\hat{\h}$ defined in (\ref{hat-h}).
Consider the following subalgebras of $\hat{\h}[\nu]$
$$
\hat{\h}[\nu]^+=\coprod_{n>0}{\h}_{(kn)}\otimes t^{n},\ \ \quad
\hat{\h}[\nu]^-=\coprod_{n<0}{\h}_{(kn)}\otimes t^{n}
$$
and 
$$
\hat{\h}[\nu]_{\frac{1}{k} \Z}=\hat{\h}[\nu]^+\oplus
\hat{\h}[\nu]^-\oplus \C {\bf k},
$$
which is a Heisenberg subalgebra of $\hat{\h}[\nu]$.  As in
(\ref{induced-h}) form the induced $\hat{\h}[\nu]$-module
\begin{equation} \label{S[nu]}
S[\nu]=U(\hat{\h}[\nu])\otimes_{U(\coprod_{n \ge 0}{\h}_{(kn)}\otimes
t^{n}\oplus \C {\bf k})} \C,
\end{equation}
where $\coprod_{n\ge 0} \h_{(kn)} \otimes t^{n}$ acts trivially on
$\C$ and ${\bf k}$ acts as 1.  This is an irreducible
$\hat{\h}[\nu]_{\frac{1}{k} \Z}$-module, which is linearly isomorphic
to the symmetric algebra $S(\hat{\h}[\nu]^-)$. As in Section 6 of
\cite{DL1} we give the module $S[\nu]$ the natural
$\mathbb{Q}$-grading by weights compatible with the action of
$\hat{\frak{h}}[\nu]$ and such that
\begin{equation} \label{weight-1a} 
\mbox{wt} \; 1= \frac{1}{4k^2}
\sum_{j=1}^{k-1} j(k-j) \mbox{dim} \; \frak{h}_{(j)}
\end{equation}
(see also formula (2.20) in \cite{BHL}).  The reason for choosing this
shifted grading will be justified later by the action of the operator
$L^{\hat{\nu}}(0)$.

Continuing to follow \cite{L1}, denote by $N$ the orthogonal
complement of $\h_{(0)}$ in $L$:
\begin{equation} \label{N}
  N= (1-P_{0}) \h \cap L = \{ \alpha \in L \; | \; \langle \alpha, 
\h_{(0)} \rangle =0 \}.
\end{equation}
Let
\begin{equation} \label{M}
M=(1-\nu) L \subset N.
\end{equation}
Since $\sum_{j=0}^{k-1} \nu^j \alpha \in \h_{(0)}$ for any $\alpha \in
\h$, the commutator map (\ref{commutator-definition}) becomes
\begin{equation} \label{C_N}
  C_N(\alpha, \beta)= \eta^{\sum_{j=0}^{k-1} 
\langle j \nu^j \alpha, \beta \rangle}
\end{equation}
for $\alpha, \beta \in N$. Let
\begin{equation} \label{R}
R= \{ \alpha \in N \; | \: C_N(\alpha, N)=1 \}.
\end{equation}
Note that $M \subset R$.  For any subgroup $Q$ of $L$ we denote by
$\hat{Q}$ the subgroup of $\hat{L}_{\nu}$ obtained by pulling back
$Q$. By Proposition 6.1 of \cite{L1}, there exists a
unique homomorphism $\tau: \hat{M} \rightarrow
\mathbb{C}^{\times}$ such that 
$$
\tau(\eta_0) = \eta_0 \ \ \mbox{and}\ \
\tau(a\hat{\nu}a^{-1}) = \eta^{-\sum_{j=0}^{k-1} 
\langle \nu^j \overline{a}, \overline{a}
\rangle/2}
$$
for $a \in \Lnu$. 

Let us now recall the classification of the irreducible
$\hat{N}$-modules:

\begin{prop} (Proposition 6.2 of \cite{L1}) \label{classification}
  There are exactly $|R/M|$ extensions of $\tau$ to a homomorphism
  $\chi: \hat{R} \rightarrow \mathbb{C}^{\times}$. For each $\chi$,
  there is a unique (up to equivalence) irreducible $\hat{N}$-module
  on which $\hat{R}$ acts according to $\chi$, and every irreducible
  $\hat{N}$-module on which $\hat{M}$ acts according to $\tau$ is
  equivalent to one of these. Every such module has dimension
  $|N/R|^2$.
\end{prop}

Let $T$ be an irreducible $\hat{N}$-module. Form the induced
$\hat{L}_{\nu}$-module
\begin{equation} \label{U_T}
U_T = \C[\Lnu] \otimes_{\C[\hat{N}]} T \simeq \C[L/N] \otimes T,
\end{equation}
where $\Lnu$ and $\h_{(0)}$ act as follows:
\begin{eqnarray} \label{action-Lnu}
a \cdot b \otimes u &=& ab \otimes u, \\ \label{action-h}
h \cdot b \otimes u &=& \langle h, \overline{b} \rangle b 
\otimes u
\end{eqnarray}
for $a,b \in \Lnu$, $u \in T$, $h \in \h_{(0)}$.  As operators on $U_T$,
\begin{equation}\label{h-operates}
ha=a(\langle h,\overline{a} \rangle+h).
\end{equation}
For $h \in \h_{(0)}$ define
the $\mbox{End} \, U_T$-valued formal Laurent series $x^h$ by:
\begin{equation} \label{x^h}
x^h \cdot b \otimes u= x^{\langle h, \overline{b} \rangle} b \otimes u.
\end{equation} 
Now for $h \in \h_{(0)}$ such that $\langle h, L \rangle \subset
\mathbb{Z}$ define the operator $\eta^h$ on $U_T$ as follows:
\begin{equation}\label{action-eta}
\eta^h \cdot b \otimes u= \eta^{\langle h, \overline{b} 
\rangle} b \otimes u.
\end{equation}
Then for $a \in \hat{L}_{\nu}$ we have 
\begin{equation}\label{x-a}
  x^h a= ax^{\langle h, \overline{a} \rangle + h} \ \; \mbox{for} \; \; 
h \in \h_{(0)}
\end{equation}
and
\begin{equation} \label{eta-a}
\eta^h a = \eta^{\langle h, \overline{a} \rangle +h}a \; \; \mbox{for} \; \; 
h \in \h_{(0)} \; \; \mbox{with} \; \; 
\langle h, L \rangle \in \mathbb{Z}.
\end{equation}
Moreover, as operators on $U_T$,
\begin{equation}
\hat{\nu} a = a \eta^{-\sum_{j=0}^{k-1} \nu^j \overline{a} -\sum_{j=0}^{k-1}
\langle \nu^j \overline{a}, \overline{a} \rangle/2} 
\end{equation}
for $a \in \hat{L}_{\nu}$.

Since the projection map $P_0$
(see (\ref{Pp})) induces an isomorphism {}from $L/N$ onto $P_{0}L,$
we have a natural isomorphism 
\begin{equation} \label{U_T-P_0L}
U_T \simeq \C[P_{0}L] \otimes T
\end{equation}
of $\hat{\h}[\nu]$-modules.
We also have
$$
U_T =\coprod_{\alpha \in P_0L} U_{\alpha},
$$
where 
$$
U_{\alpha}= \{ u \in U_T \; | \; 
h \cdot u = \langle h, \alpha \rangle u 
\; \; \mbox{for} \; \; h \in \frak{h}_{(0)} \}
$$ 
and 
$$
a \cdot U_{\alpha} \subset U_{\alpha +\overline{a}_{(0)}}
\; \; \; \mbox{for} \; \; a \in \hat{L}_{\nu}, \; \alpha \in P_0L. 
$$
Consider the $\mathbb{C}$-grading on $U_T$ given by
\begin{equation} \label{grading-U_T}
\mbox{wt} \; u= \frac{1}{2} \langle \alpha, \alpha \rangle
 \; \; \; \mbox{for} \; \; \; 
u \in U_{\alpha}, \; \alpha \in P_0L.  
\end{equation}

Set 
\begin{eqnarray} \label{V_L^T}
V^T_L &=& S[\nu]\otimes U_T \\ &=&
\left(U(\hat{\h}[\nu])\otimes_{U(\coprod_{n \ge 0}{\h}_{(kn)}\otimes
t^{n}\oplus \C {\bf c})} \C\right)
\otimes \left(\C[\Lnu]  \otimes_{\C[\hat{N}]} T \right) \nonumber \\
&\simeq & S(\hat{\h}[\nu]^-) \otimes \C[P_{0}L] \otimes T, \nonumber
\end{eqnarray}
on which $\Lnu,$ $\hat{\h}[\nu]_{\frac{1}{k} \Z},$ ${\h}_{(0)}$ and
$x^h$ for $h\in{\h}_{(0)}$ act naturally on either $S[\nu]$ or $U_T$
as described above.  The space $V_L^T$ is graded by weights using the
weight gradings of $S[\nu]$ and $U_T$, as described above.

For $\alpha \in \h$ and $n\in \frac{1}{k} \Z$, write
$\alpha^{\hat{\nu}} (n)$ or $\alpha_{(kn)} (n)$ for the operator on
$V_L^T$ associated with $\alpha_{(kn)}\otimes t^n \in \hat{\h}[\nu]$:
\begin{equation} \label{alpha-operator}
\alpha_{(kn)} \otimes t^n \mapsto \alpha^{\hat{\nu}}(n),
\end{equation}
and set
\begin{equation} 
\alpha^{\hat{\nu}} (x)=\sum_{n\in\frac{1}{k} \Z} \alpha^{\hat{\nu}}
(n)x^{-n-1} =\sum_{n\in\frac{1}{k} \Z} \alpha_{(kn)} (n) x^{-n-1}.
\end{equation}
Consider the formal Laurent series $E^{\pm}(\alpha, x) \in (\mbox{End}
\, V_L)[[x^{1/k}, x^{-1/k}]] $ (recall ({\ref{exp})). We have
  \begin{equation} \label{e-series} E^{+}(\alpha, x_1)E^{-}(\beta,
    x_2)=E^{-}(\beta, x_2)E^{+}(\alpha, x_1) \prod_{p \in \Z/k\Z}
    \left ( 1- \eta^p \frac{{x_2}^{1/k}}{x_1^{1/k}} \right )^{\langle
      \nu^p \alpha, \beta \rangle }
\end{equation}
for $\alpha, \beta \in \frak{h}$.
Let
\begin{equation} \label{sigma}
\sigma(\alpha) = \left\{ \begin{array}{ll}
\displaystyle
{ \prod_{0< j < k/2} (1-\eta^{-j})^{\langle \nu^j \alpha, \alpha
\rangle} } 2^{\langle \nu^{k/2} \alpha, \alpha \rangle/2}& \mbox{if $k \in
2\mathbb{Z}$} \\
\\
\displaystyle
{\prod_{0 < j < k/2  } (1-\eta^{-j})^{\langle \nu^j \alpha, \alpha
\rangle} } & \mbox{if $k \in 2\mathbb{Z} + 1.$}\\
\end{array}
\right.
\end{equation}
Now for $a \in \hat{L}$ define the $\hat{\nu}$-twisted vertex operator
$Y^{\hat{\nu}}(\iota(a), x)$ acting on $V_L^T$ as follows:
\begin{equation}\label{L-operator}
Y^{\hat{\nu}}(\iota(a),x)= k^{-\langle \overline{a},\overline{a}
\rangle /2} \sigma(\overline{a}) \ _\circ^\circ e^{\sum_{n \neq
0}\frac{-\overline{a}(n)}{n}x^{-n}} \ _\circ^\circ a
x^{\overline{a}_{(0)}+\langle \overline{a}_{(0)} ,\overline{a}_{(0)}
\rangle /2-\langle \overline{a}, \overline{a} \rangle /2},
\end{equation}
where we view $a$ in the right hand side of (\ref{L-operator}) as an
element of $\hat{L}_{\nu}$ using the set-theoretic identification
between $\hat{L}$ and $\hat{L}_{\nu}$ given in
(\ref{identify-central-extensions}). By using (\ref{exp}) we have
\begin{equation} 
Y^{\hat{\nu}}(\iota(a), x)= k^{-\langle \overline{a},
\overline{a}\rangle /2} 
\sigma(\overline{a}) E^{-} (-\overline{a}, x) E^{+} (-\overline{a}, x)
a x^{\overline{a}_{(0)}+\langle
\overline{a}_{(0)},
\overline{a}_{(0)}\rangle /2-\langle \overline{a},
\overline{a}\rangle /2}.
\end{equation}
Define the component operators $x_{\alpha}^{\hat{\nu}}(n)$ for $n \in (1/k)
\mathbb{Z}$ and $\alpha \in L$ by the expansion 
\begin{equation}
Y^{\hat{\nu}} (\iota(e_{\alpha}), x)= \sum_{n \in (1/k)\mathbb{Z}}
x_{\alpha}^{\hat{\nu}}(n) x^{-n -\frac{\langle \alpha, \alpha \rangle}{2}}.
\end{equation}

For $v=\alpha_1(-n_1)\cdots \alpha_m(-n_m) \cdot \iota(a)\in V_L$, set
\begin{equation}
W(v,x)= \ _\circ^\circ
\biggl
(\frac{1}{(n_1-1)!}\left(\frac{d}{dx}\right)^{n_1-1}\alpha_1^{\hat{\nu}}(x)\biggr)\cdots
\left(\frac{1}{(n_m-1)!}\biggl(\frac{d}{dx}\right)^{n_m-1}
\alpha_m^{\hat{\nu}}(x)\biggr)Y^{\hat{\nu}}(\iota(a),x) \ _\circ^\circ ,
\end{equation}
giving a well-defined linear operator on $V^T_L$ depending linearly on
$v\in V_L$ (as in (\ref{untwisted-operator})).

Define constants $c_{mnr} \in \C$ for $m, n \in \mathbb{N}$ and $r =
0,\dots, k-1$ by
\begin{eqnarray}  \label{constants-0}
\sum_{m,n\ge 0} c_{mn0} x^m y^n &=& -\frac{1}{2} \sum_{j = 1}^{k-1}
{\rm log} \left(\frac {(1+x)^{1/k} - \eta^{-j}
(1+y)^{1/k}}{1-\eta^{-j} }\right),\\ 
\sum_{m,n\ge 0} c_{mnr} x^m y^n &= & \frac{1}{2}{\rm log} \left( \frac
{(1+x)^{1/k} -\eta^{-r}
(1+y)^{1/k}}{1-\eta^{-r}}\right) \ \ \mbox{for}\ \ r \ne 0 \label{constants-r}
\end{eqnarray}
(well-defined formal power series in $x$ and $y$).  Let
$\{\beta_1,\dots, \beta_{\dim \h}\}$ be an orthonormal basis of $\h$,
and set
$$
\Delta_x = \sum_{m,n\ge 0} \sum_{r=0}^{k-1} \sum^{\dim \h}_{j=1}
c_{mnr} (\nu^{-r} \beta_j)(m) \beta_j(n) x^{-m-n}.
$$
Note that $\Delta_x$ is independent of the choice of the othonormal
basis. Then $e^{\Delta_x}$ is well-defined on $V_L$ since $c_{00r}=0$
for all $r$, and for $v\in V_L$ we have $e^{\Delta_x}v\in
V_L[x^{-1}]$.

Now for $v\in V_L,$ we define the $\hat{\nu}$-$twisted$ $vertex$
$operator$
\begin{equation}\label{Ynuhat}
Y^{\hat{\nu}}(v,x)=W(e^{\Delta_x}v,x)
\end{equation}
and this yields a well-defined linear map
\begin{eqnarray}
V_L &\longrightarrow&(\mbox{End}\,V^T_L)[[x^{1/k},x^{-1/k}]] \\ \ v
&\mapsto& Y^{\hat{\nu}}(v,x)= 
\sum_{n \in \frac{1}{k}\Z}v^{\hat{\nu}}_nx^{-n-1},
\; \; v^{\hat{\nu}}_n\in {\rm End}\,V^T_L. \nonumber
\end{eqnarray}

By \cite{FLM2}, \cite{FLM3} and \cite{L2} (see also
\cite{DL1}), $V_L^T=(V_L^T, Y^{\hat{\nu}})$ has the structure of an
irreducible $\hat{\nu}$-twisted $V_L$-module. In particular, we have
the twisted Jacobi identity
\begin{multline} \label{Jacobi}
x^{-1}_0\delta\left(\frac{x_1-x_2}{x_0}\right)
Y^{\hat{\nu}}(u,x_1)Y^{\hat{\nu}}(v,x_2)-x^{-1}_0
\delta\left(\frac{x_2-x_1}{-x_0}\right) Y^{\hat{\nu}}
(v,x_2)Y^{\hat{\nu}} (u,x_1)\\  
= x_2^{-1}\frac{1}{k}\sum_{j\in \Z /k \Z}
\delta\left(\eta^j\frac{(x_1-x_0)^{1/k}}{x_2^{1/k}}\right)Y^{\hat{\nu}} 
(Y(\hat{\nu}^j
u,x_0)v,x_2) 
\end{multline}
for $u, v \in V_L$ (the main property of a twisted module), and also,
the $\hat{\nu}$-twisted operator has the property
\begin{equation} \label{Y-hat}
Y^{\hat{\nu}}(\hat{\nu} v, x)= 
\lim _{x^{1/k} \rightarrow \eta^{-1} x^{1/k}}
Y^{\hat{\nu}}(v, x)
\end{equation}
for $v \in V_L$. Formula (\ref{Y-hat}) immediately generalizes to
\begin{equation}\label{Y-hat-r}
Y^{\hat{\nu}}(\hat{\nu}^r v, x)= 
\lim _{x^{1/k} \rightarrow \eta^{-r} x^{1/k}}
Y^{\hat{\nu}}(v, x)
\end{equation}
for any $r \in \Z$. By taking $\mbox{Res}_{x_0}$, the twisted Jacobi
identity (\ref{Jacobi}) immediately implies the commutator
formula \cite{FLM2}:
\begin{equation} \label{commutator}
[ Y^{\hat{\nu}}(u, x_1), Y^{\hat{\nu}}(v, x_2) ]  
= x_2^{-1}\frac{1}{k} \mbox{Res}_{x_0} \left ( \sum_{j\in \Z /k \Z}
\delta\left(\eta^j\frac{(x_1-x_0)^{1/k}}{x_2^{1/k}}\right)Y^{\hat{\nu}}
(Y(\hat{\nu}^j u,x_0)v,x_2) \right ). 
\end{equation}

Following \cite{DL1} (see also \cite{BHL}) we will now justify that
the weight grading of $V_L^T$ given by (\ref{grade-h}),
(\ref{weight-1a}) and (\ref{grading-U_T}) is the grading given by the
operator $L^{\hat{\nu}}(0)$, where the operators $L^{\hat{\nu}}(n)$
for $n \in \Z$ are defined by
$$
Y^{\hat{\nu}}(\omega, x)=\sum_{n \in \Z} 
L^{\hat{\nu}} (n) x^{-n-2}
$$
(recall (\ref{omega})).  These operators have the property
$$
[L^{\hat{\nu}}(m), L^{\hat{\nu}}(n)]= (m-n) L^{\hat{\nu}}(m+n) 
+ \frac{1}{12} (m^3-m) \delta_{m+n, 0} \mbox{dim} \; \frak{h}
$$
for $m,n \in \Z$.  By Proposition 6.3 of \cite{DL1} we have
\begin{eqnarray} \label{important}
& [Y^{\hat{\nu}}(\omega, x_1), Y^{\hat{\nu}}(\iota(a), x_2) ] & \\
&=x_2^{-1} (\frac{d}{dx_2} Y^{\hat{\nu}}(\iota(a), x_2)) \delta(x_1/x_2)-
\frac{1}{2} \langle \overline{a}, \overline{a} \rangle 
x_2^{-1} Y^{\hat{\nu}}(\iota(a), x_2) \frac{\partial}{\partial x_1} \delta
(x_1/x_2)& \nonumber
\end{eqnarray}
for $a \in \hat{L}$.  Also recall from \cite{DL1}, \cite{BHL} and \cite{DLeM} that
\begin{equation} \label{grading-1}
L^{\hat{\nu}} (0) 1 = 
\frac{1}{4k^2} \sum_{j=1}^{k-1} j(k-j) \mbox{dim} \; \frak{h}_{(j)}1
\end{equation}
($1 \in S[\nu]$),
\begin{equation} \label{grading-2}
L^{\hat{\nu}}(0) u= \left (\frac{1}{2} \langle \alpha, \alpha \rangle +
\frac{1}{4k^2} \sum_{j=1}^{k-1} j(k-j) \mbox{dim} \; \frak{h}_{(j)} \right )u
\end{equation}
for $u \in U_{\alpha} \subset U_T \subset V_L^T$ ($\alpha \in P_0L$), and
\begin{equation} \label{grading-3}
[L^{\hat{\nu}}(0), \alpha^{\hat{\nu}}(m)]=-m \alpha^{\hat{\nu}}(m)
\end{equation}
for $m \in \frac{1}{k} \Z$ and $\alpha \in \frak{h}_{(km)}$. Thus by
using the grading shift (\ref{weight-1a}) and the weight grading
defined by (\ref{grade-h}) and (\ref{grading-U_T}) we have
\begin{equation} \label{grading-4}
L^{\hat{\nu}}(0) v= \left (\mbox{wt} \; v + \frac{1}{4k^2} \sum_{j=1}^{k-1} j(k-j) 
\mbox{dim} \; \frak{h}_{(j)} \right ) v 
\end{equation}
for a homogenous element $v \in V_L^T$. 

It has been established in \cite{L1} (see also \cite{FLM2} and
\cite{FLM3}) that if the even lattice $L$ is the root lattice of a Lie
algebra of type $A$, $D$ or $E$ then $V_L^T$ has a natural structure
of module for a certain twisted affine Lie algebra. In the next
section we will recall in detail the special case
that for $L$ the root lattice of $\frak{sl}(3, \C)$
and for a certain isometry of this root lattice the corresponding
twisted module $V_L^T$ is an irreducible $A_2^{(2)}$-module.

\section{Vertex operator construction of  $A_2^{(2)}$}

\setcounter{equation}{0} The aim of this section is to recall the
twisted vertex operator construction of the affine Lie algebra
$A_2^{(2)}$ as a special case of the lattice construction recalled in
the previous section, following the treatment in \cite{L1} and
\cite{FLM2}. We will specialize the previous section to the root
lattice of $\frak{sl}(3, \C)$ and an involution $\nu$ induced by a
Dynkin diagram automorphism of $\frak{sl}(3, \C)$.

Let $\frak{h}$ be a Cartan subalgebra of $\frak{sl}(3, \C)$.  Denote
by $\Delta \subset \frak{h}^{*}$ the root system and by $ \{ \alpha_1,
\alpha_2 \}$ a choice of simple roots. Take $\langle a, b
\rangle=\mbox{tr} (ab)$ ($a, b \in \frak{sl}(3, \C)$), the standard
suitably-normalized nonsingular symmetric invariant bilinear form on
$\frak{sl}(3, \C)$. We identify $\frak{h}$ with $\frak{h}^{*}$ via
$\langle \cdot, \cdot \rangle$, so that under this identification we
have $\Delta \subset \frak{h}$ (and $\alpha_1, \alpha_2 \in
\frak{h}$).

We now specialize the previous section to the root lattice of
$\frak{sl}(3, \C)$,
\begin{equation} \label{root-lattice}
L=\mathbb{Z} \Delta= \mathbb{Z}\alpha_1 \oplus \mathbb{Z} \alpha_2
\subset \frak{h},
\end{equation}
equipped with the form $\langle \cdot, \cdot \rangle$. We take $\nu$
to be the isometry of $L$ determined by
\begin{equation} \label{twist}
\nu (\alpha_1)= \alpha_2 , \; \; \nu (\alpha_2)=\alpha_1,
\end{equation}
corresponding to the Dynkin diagram automorphism.  Although $\nu^2=1$,
we take $k=4$ rather than $2$ as our period of $\nu$, since otherwise
the assertion (\ref{even-equation}) would not hold. Then we have
(\ref{even-equation}) and (\ref{even-condition}) with $k=4$. Fix
\begin{equation}
\eta=i
\end{equation}
to be our primitive $4^{\mbox{\rm th}}$ root of unity. With
$\eta_{0}=(-1)^4 \eta$ as in (\ref{eta-0}), we in fact have
\begin{equation}
\eta_0=\eta=i.
\end{equation}
Extend $\nu$ linearly to an automorphism of
\begin{equation} \label{cartan}
\frak{h}= \mathbb{C} \otimes_{\mathbb{Z}}L,
\end{equation}
our Cartan subalgebra.

We have the two central extensions of $L$ by the cyclic group
generated by $i$, $\hat{L}$ and $\hat{L}_{\nu}$, with $C_0$ and $C$,
respectively, the commutator maps (recall
(\ref{commutator-definition-0}), (\ref{commutator-definition}) and
(\ref{exact-0})--(\ref{commutator=C})). As before, we choose the
normalized sections $e$ of $\hat{L}$ and $\hat{L}_{\nu}$ that send
$\alpha \in L$ to $e_{\alpha} \in \hat{L}$ (respectively,
$\hat{L}_{\nu}$). We also have the normalized cocycles $\epsilon_C$
and $\epsilon_{C_0}$ (see (\ref{epsilon_c}) and
(\ref{epsilon_c_0})). By (\ref{prod-C_0}) we get
$$
e_{\alpha_1} e_{\alpha_2} = \epsilon_{C_0} (\alpha_1, \alpha_2) 
e_{\alpha_2+\alpha_1} \; \; \; \mbox{in} \; \; \; \hat{L},
$$
and since 
$$
\epsilon_{C_0}(\alpha_1, \alpha_2)/
\epsilon_{C_0}(\alpha_2, \alpha_1)=C_0(\alpha_1, \alpha_2)
=-1
$$ 
by (\ref{quotient-C_0}) and (\ref{commutator-definition-0}), we have
\begin{equation} \label{multiplication-hat-L}
e_{\alpha_1}e_{\alpha_2}=-e_{\alpha_2}e_{\alpha_1} \; \; \; \mbox{in} 
\; \; \; \hat{L}.
\end{equation}

For concreteness and convenience we shall use the following particular
choice of $\epsilon_{C_0}$: Take $\epsilon_{C_0}: L \times L
\longrightarrow \langle i \rangle$ to be the $\Z$-bilinear map
determined by
\begin{equation} \label{epsilon(1,2)}
\epsilon_{C_0} (\alpha_1, \alpha_2)=1, \; \; \; 
\epsilon_{C_0}(\alpha_2, \alpha_1)=-1
\end{equation}
and 
\begin{equation}
\epsilon_{C_0}(\alpha_1, \alpha_1)=\epsilon_{C_0}(\alpha_2, \alpha_2)=1.
\end{equation}
Then, in particular,
\begin{equation} \label{epsilon(1,-1)}
\epsilon_{C_0}(\alpha,-\alpha)=1 \; \; \; \mbox{for} \; \; \; \alpha
=\alpha_1 \; \; \mbox{or} \; \; \alpha_2.
\end{equation}
We have that (\ref{epsilon-1})--(\ref{epsilon-3}) hold for
$\epsilon_{C_0}$, that is, $\epsilon_{C_0}$ is a normalized $2$-cocycle
associated with the commutator map $C_0$. This $2$-cocycle has the
properties
\begin{equation} \label{prop-1}
\epsilon_{C_0}(\alpha, \beta)^2=1
\end{equation}
and
\begin{equation} \label{prop-2}
\epsilon_{C_0}(\alpha, \beta) = \epsilon_{C_0}(\nu \beta, \nu \alpha)
\end{equation}
for any $\alpha, \beta \in L$. Indeed, for $\alpha=m\alpha_1+n
\alpha_2$ and $\beta=r\alpha_1+s\alpha_2$ with $m,n, r, s \in \Z$,
we have 
$$
\epsilon_{C_0}(\alpha, \beta)= (-1)^{nr}= \epsilon_{C_0}(\nu \beta,
\nu \alpha),
$$
which gives (\ref{prop-1})--(\ref{prop-2}).
 
As in (\ref{nu-on-hatL-0})--(\ref{nu-on-hatL}), we lift the isometry
(\ref{twist}) of $L$ to an automorphism $\hat{\nu}$ of $\hat{L}$ and
of $\hat{L}_{\nu}$ fixing $i$ and satisfying (\ref{nuhata=a}), so that
$\hat{\nu}^4=1$ (recall (\ref{nuhat^k=1})). Again for concreteness, we
make the following particular choice of $\hat{\nu}$: 
\begin{equation} \label{particular-nu}
\hat{\nu} e_{\alpha}= \epsilon_{C_0}(\alpha, \alpha) i^{\langle
 \alpha, \alpha_1+\alpha_2 \rangle} e_{\nu \alpha}
\end{equation}
for $\alpha \in L$. Then $\hat{\nu}$ is indeed an automorphism of
$\hat{L}$, since for any $\alpha, \beta \in L$ we obtain
\begin{equation} \label{lhs} 
\hat{\nu}(e_{\alpha}e_{\beta}) = \epsilon_{C_0}(\alpha, \alpha) \epsilon_{C_0}
(\beta, \beta) \epsilon_{C_0}(\alpha, \beta)^2 \epsilon_{C_0}(\beta, \alpha)
i^{\langle \alpha+\beta, \alpha_1+\alpha_2 \rangle } e_{\nu \alpha + \nu \beta}
\end{equation}
and
\begin{equation} \label{rhs}
(\hat{\nu} e_{\alpha}) (\hat{\nu} e_{\beta})=
\epsilon_{C_0}(\alpha, \alpha) \epsilon_{C_0}(\beta, \beta) \epsilon_{C_0}
(\nu \alpha, \nu \beta) 
i^{\langle \alpha+\beta, \alpha_1+\alpha_2 \rangle } e_{\nu \alpha + \nu \beta},
\end{equation}
and by (\ref{prop-1}) and (\ref{prop-2}) we observe that (\ref{lhs})
and (\ref{rhs}) are equal. Using (\ref{identify-central-extensions}),
(\ref{epsilon_c_0-def}) and the fact $\hat{\nu}$ is an automorphism of
$\hat{L}$ we obtain that $\hat{\nu}$ is an automorphism of
$\hat{L}_{\nu}$. We also have that (\ref{particular-nu}) is a lifting
of (\ref{twist}) and that it satisfies (\ref{nuhata=a}). Since
$$
\hat{\nu}^2 e_{\alpha}= \epsilon_{C_0} (\alpha, \alpha) \epsilon_{C_0}
(\nu \alpha, \nu \alpha) i^{\langle \alpha+\nu \alpha, \alpha_1+\alpha_2 
\rangle} e_{\alpha},
$$
by (\ref{prop-1}), (\ref{prop-2}) and the fact
that $i^{\langle \alpha+\nu \alpha, \alpha_1+\alpha_2 \rangle}=-1$ we
obtain $\hat{\nu}^2e_{\alpha}=-e_{\alpha}$ for any $\alpha \in
L$. This confirms that
\begin{equation} \label{particular-nu^4=1}
\hat{\nu}^4=1,
\end{equation}
but note that $\hat{\nu}^2 \neq 1$:
$$
\hat{\nu}^2=-1.
$$
Formula (\ref{particular-nu}) yields in particular
\begin{equation} \label{nu-alpha}
\hat{\nu} e_{\alpha_1}=ie_{\alpha_2}, \; \; \; 
\hat{\nu} e_{\alpha_2}=i e_{\alpha_1}
\end{equation}
(and these two formulas determine the automorphism $\hat{\nu}$
uniquely),
\begin{equation} \label{nu-alpha-}
\hat{\nu} e_{\alpha_1+\alpha_2}=e_{\alpha_1+\alpha_2},
\end{equation}
and
\begin{equation} \label{-nu}
\hat{\nu} e_{-\alpha_1}=-i e_{-\alpha_2}, \; \; \; \hat{\nu} e_{\alpha_2}
=-ie_{-\alpha_1}, \; \; \; \hat{\nu} e_{-\alpha_1-\alpha_2}=e_{
-\alpha_1-\alpha_2}.
\end{equation}

Recall from the previous section the construction of the vector space
$V_L$ (\ref{voa}), which together with the vertex operator $Y(\cdot,
x)$ (\ref{untwisted-operator}), a vacuum vector and a conformal vector
forms a vertex operator algebra that has a natural
$\mathbb{Z}$-grading by weights. Following
(\ref{extension-one})--(\ref{extension-final}) we extend the
automorphism $\hat{\nu}$ of $\hat{L}$ given by (\ref{particular-nu})
to an automorphism of $V_L$ denoted by $\hat{\nu}$ as well. This acts
via $\nu \otimes \hat{\nu}$, preserves the grading and has period $4$.

For $n \in \mathbb{Z}$ set
\begin{equation}
\frak{h}_{(n)}= \{ x \in \frak{h} \; | \; \nu(h)= i^n h \} \subset \frak{h},
\end{equation}
such that
\begin{equation}
\frak{h}= \coprod_{n \in \mathbb{Z}/4 \mathbb{Z}}\frak{h}_{(n)}.
\end{equation}
We identify $\frak{h}_{(n \; \rm{mod} \; 4)}$ with
$\frak{h}_{(n)}$. In view of (\ref{twist}) extended linearly to 
$\frak{h}$ we have:
\begin{equation} \label{h_0}
\frak{h}_{(0)}= \{ s(\alpha_1+\alpha_2) | s \in \mathbb{C} \},
\end{equation}
\begin{equation} \label{h_2}
\frak{h}_{(2)} = \{ s (\alpha_1-\alpha_2)  | s \in \mathbb{C} \}
\end{equation}
and 
\begin{equation} \label{h_1}
\frak{h}_{(1)}=\frak{h}_{(3)}=0.
\end{equation}
Using the $n^{\mbox{\rm th}}$ projection (\ref{Pp}) set $h_{(n)}=P_{n
\; \rm{mod} \; 4} h$ for $h \in \frak{h}$ and $n \in \mathbb{Z}$. For
$j=1,2$, we have
\begin{equation} \label{oper-1}
({\alpha_{j}})_{(0)}= \frac{1}{2} (\alpha_1+\alpha_2), \; \; 
({\alpha_j})_{(2)}=\frac{1}{2} (\alpha_1-\alpha_2)
\end{equation}
and
\begin{equation} \label{oper-2}
({\alpha_j})_{(1)}= ({\alpha_j})_{(3)}=0.
\end{equation}

Form the $\nu$-twisted affine Lie algebra associated to the abelian
Lie algebra $\frak{h}$:
\begin{equation} \label{twisted}
\hat{\frak{h}}[\nu] = \coprod _{n \in \mathbb{Z}} 
\frak{h}_{(n)} \otimes t^{\frac{n}{4}} \oplus \mathbb{C} {\bf k}
=\coprod _{n \in \frac{1}{4}\mathbb{Z}} 
\frak{h}_{(4n)} \otimes t^{n} \oplus \mathbb{C} {\bf k}
\end{equation}
such that ${\bf k}$ is a central element and 
$$
[\alpha \otimes t^m, \beta \otimes t^n]=\langle \alpha, \beta \rangle m 
\delta_{m+n, 0}{\bf k}
$$
for $m, n \in \frac{1}{4} \mathbb{Z}$ and $\alpha \in \frak{h}_{(4m)},
\; \beta \in \frak{h}_{(4n)}$. By (\ref{h_0})--(\ref{h_1}) we have
\begin{eqnarray} \label{new twisted} 
\hat{\frak{h}}[\nu] & = & \coprod
_{n \in \mathbb{Z}} \frak{h}_{(0)} \otimes t^n \oplus \coprod_{n \in
\mathbb{Z} + \frac{1}{2}} \frak{h}_{(2)} \otimes t^n \oplus
\mathbb{C}{\bf k} \\ \nonumber & = & \frak{h}_{(0)} \otimes \mathbb{C}
[t, t^{-1}] \oplus \frak{h}_{(2)} \otimes t^{1/2} \mathbb{C} [t,
t^{-1}] \oplus \mathbb{C} {\bf k} \\ \nonumber
\end{eqnarray}
This algebra is $\frac{1}{2}\mathbb{Z}$-graded by weights:
$$
\mbox{wt} (\alpha \otimes t^m)=-m, \; \; \mbox{wt} \; {\bf k}=0
$$
for $m \in (1/2) \mathbb{Z}$ and $\alpha \in
\frak{h}_{(4m)}$. Consider the Heisenberg subalgebra of
$\hat{\frak{h}}[\nu]$,
$$
\hat{\frak{h}}[\nu]_{\frac{1}{4} \Z}=\coprod _{n \in \frac{1}{4} \Z, \; 
n \neq 0} 
\frak{h}_{(4n)} \otimes t^{n} \oplus \mathbb{C}{\bf k}
$$
and the subalgebras
$$
\hat{\frak{h}}[\nu]^{\pm}= \coprod_{n \in \frac{1}{4} \Z, \; 
\pm n>0}\frak{h}_{(4n)} \otimes t^{n}.
$$
The induced $\hat{\frak{h}}[\nu]$-module (\ref{S[nu]}) becomes
\begin{equation}
S[\nu]=U(\hat{\frak{h}}[\nu])\otimes_{U(\coprod_{n \geq 0}
{\frak{h}}_{(4n)}\otimes
t^{n}\oplus \C {\bf k})} \C \simeq S(\hat{\frak{h}}[\nu]^-),
\end{equation}
and this is irreducible as an
$\hat{\frak{h}}[\nu]_{\frac{1}{4}\mathbb{Z}}$-module. The module
$S[\nu]$ is $\Q$-graded such that
\begin{equation}
\mbox{wt} \; 1= \frac{1}{16}
\end{equation}
by (\ref{weight-1a}).

Recall from the previous section the spaces $N$, $M$ and $R$
(see (\ref{N}), (\ref{M}) and (\ref{R})). In our setting we have
$$
N=M=\{ s(\alpha_1-\alpha_2) \; | \; s \in \Z \}
$$
and 
$$
C_N(\alpha, \beta)=1 \; \; \mbox{for} \; \; \alpha, \beta \in N
$$
(cf. (\ref{C_N})). Thus
\begin{equation}
N=M=R,
\end{equation}
and so,
\begin{equation}
\hat{N}=\hat{M}=\hat{R}.
\end{equation}
(Here we use our notation for pulling back a subgroup of $L$
introduced in Section 2.)  By Proposition 6.1 in \cite{L1} there is a
unique homomorphism $\tau: \hat{M}=\hat{N} \rightarrow
\mathbb{C}^{\times}$ such that
$$
\tau(i)=i, \; \; \tau(a \hat{\nu} a^{-1}) =
i^{-\sum_{j=0}^3 \langle
\nu^j \bar{a}, \bar{a} \rangle/2}.
$$
Denote by $\mathbb{C}_{\tau}$ the one-dimensional $\hat{N}$-module
$\mathbb{C}$ with character $\tau$ and write
\begin{equation} \label{T}
T= \mathbb{C}_{\tau}.
\end{equation}
This is the unique (up to equivalence) irreducible $\hat{N}$-module
given by Proposition 6.2 in \cite{L1} (see also Proposition
\ref{classification}). The induced $\hat{L}_{\nu}$-module (\ref{U_T})
becomes
\begin{equation} \label{U_T-A_2^2}
U_T = \C[\Lnu] \otimes_{\C[\hat{N}]} T \simeq \C[L/N], 
\end{equation}
and this is graded by weights (see (\ref{grading-U_T})). There are the
natural actions of $\hat{L}_{\nu}$, $\frak{h}_{(0)}$ and $x^{h}$ for
$h \in \frak{h}_{(0)}$ on $U_T$ (see (\ref{action-Lnu}),
(\ref{action-h}) and (\ref{x^h})).

As in (\ref{V_L^T}) set
\begin{equation}
V_L^T= S[\nu] \otimes U_T \; \; (\simeq 
S(\hat{\frak{h}}[\nu]^{-}) \otimes \mathbb{C}[L/N]),
\end{equation}
on which $\hat{L}_{\nu}$,
$\hat{\frak{h}}[\nu]_{\frac{1}{4}\mathbb{Z}}$, $\frak{h}_{(0)}$ and
$x^h \; (h \in \frak{h}_{(0)})$ act. 

Using the operators $\alpha^{\hat{\nu}}(n)$ (or $\alpha_{(4n)}(n)$) on
$V_L^T$ defined in (\ref{alpha-operator}) set
\begin{equation}
\alpha^{\hat{\nu}}(x)= 
\sum_{n \in \frac{1}{4} \mathbb{Z}}\alpha^{\hat{\nu}}(n)x^{-n-1}
\end{equation}
for $\alpha \in \frak{h}$ and $n \in \frac{1}{4}\mathbb{Z}$. Note
that by (\ref{oper-1}) and (\ref{oper-2}) we have
\begin{equation} \label{alpha_j}
\alpha^{\hat{\nu}}_j(x) =\sum_{n \in \frac{1}{2} \mathbb{Z}} 
\alpha^{\hat{\nu}}_j(n)x^{-n-1}
\end{equation}
for $j=1,2$.

The normalizing factor (\ref{sigma}) becomes 
$$
\sigma(\alpha)= (1+i)^{\langle \nu \; \alpha, \alpha
 \rangle}2^{\langle \alpha, \alpha \rangle/2}
$$
for any $\alpha \in \frak{h}$. Consider the $\hat{\nu}$-twisted
vertex operator (\ref{L-operator}) acting on $V_L^T$ for $e_{\alpha}
\in \hat{L}$ and its reformulation in terms of the formal exponential
series $E^{\pm} (\cdot, x)$,
\begin{equation} \label{twisted-operator}
 Y^{\hat{\nu}}(\iota(e_{\alpha}), x)   
= 4^{-\langle \alpha,\alpha \rangle /2} 
\sigma(\alpha) E^{-} (-\alpha, x) E^{+} (-\alpha, x)
e_{\alpha}x^{\alpha_{(0)}+\langle
\alpha_{(0)}
,\alpha_{(0)}\rangle /2-\langle \alpha,\alpha \rangle /2}.
\end{equation}
As we have mentioned in the previous section $V_L^T$ together with a
vertex operator obtained in a canonical way from
(\ref{twisted-operator}) has a natural structure of a
$\hat{\nu}$-twisted module for the vertex operator algebra
$V_L$. Consider the component operators $x_{\alpha}^{\hat{\nu}}(n)$
for $n \in (1/4)\mathbb{Z}$ and $\alpha \in L$ such that
\begin{equation} \label{operators-x}
 Y^{\hat{\nu}}(\iota(e_{\alpha}), x)=
\sum_{n \in (1/4)\mathbb{Z}}x_{\alpha}^{\hat{\nu}}(n) 
  x^{-n-\frac{\langle \alpha, \alpha \rangle}{2}}.
\end{equation}

Following Section 9 of \cite{L1} (see also \cite{FK} and \cite{S}), we
now define a nonassociative algebra $(\frak{g}, [ \cdot, \cdot ])$
over $\C$ as follows:
\begin{equation} \label{lie-algebra}
\frak{g}=\frak{h} \oplus \coprod_{\alpha \in \Delta} \C x_{\alpha},
\end{equation}
with $\frak{h}$ as in (\ref{cartan}) and $\{x_{\alpha} \}_{\alpha
  \in \Delta}$ a set of symbols, such that
$$
[h, x_{\alpha}]= \langle h, \alpha \rangle x_{\alpha}=-[x_{\alpha},
h], \; \; [\frak{h}, \frak{h} ]=0,
$$
$$
[x_{\alpha}, x_{\beta}]= \left \{ \begin{array}{rcl}
\epsilon_{C_0}(\alpha, -\alpha) \alpha & \mbox{if} 
& \alpha+\beta =0\\
\epsilon_{C_0} (\alpha, \beta) x_{\alpha+\beta} & \mbox{if} &
\langle \alpha, \beta \rangle =-1\\
0 & \mbox{if} & \langle \alpha, \beta \rangle \geq 0
\end{array} \right.
$$
for $h \in \frak{h}$ and $\alpha, \beta \in \Delta$, where as in the
previous section, $\epsilon_{C_0}$ is any normalized $2$-cocycle
associated to the commutator map (\ref{commutator-definition-0}). Then
$\frak{g}$ is a Lie algebra and in fact is a copy of $\frak{sl}(3,
\C)$. The form $\langle \cdot, \cdot \rangle$ on $\frak{h}$ extends
naturally to a nonsingular symmetric invariant bilinear form on
$\frak{g}$ by:
$$
\langle h, x_{\alpha} \rangle =\langle x_{\alpha}, h \rangle =0,
$$
$$
\langle x_{\alpha}, x_{\beta} \rangle = \left \{ \begin{array}{rcl}
\epsilon_{C_0} (\alpha, -\alpha) & \mbox{if} & \alpha+\beta=0\\
0 & \mbox{if} & \alpha+\beta \neq 0.
\end{array} \right. 
$$
Using our particular choice of the normalized $2$-cocycle
$\epsilon_{C_0}$ we have
$$
[x_{\alpha}, x_{\beta}]= \left \{ \begin{array}{rcl}
\alpha & \mbox{if} 
& \alpha+\beta =0 \; \; \mbox{and} \; \; \alpha=\alpha_1 
\; \; \mbox{or} \; \; \alpha_2\\
x_{\alpha+\beta} & \mbox{if} & \alpha=\alpha_1 
\; \; \mbox{and} \; \; \beta=\alpha_2\\
0 & \mbox{if} & \langle \alpha, \beta \rangle \geq 0
\end{array} \right.
$$
and
$$
\langle x_{\alpha}, x_{\beta} \rangle = \left \{ \begin{array}{rcl}
1 & \mbox{if} & 
\alpha+\beta=0 \; \; \; \mbox{and} \; \; \; \alpha=\alpha_1 \; \; \mbox{or}
\; \; \alpha_2\\
0 & \mbox{if} & \alpha+\beta \neq 0
\end{array} \right.
$$
(see (\ref{epsilon(1,2)})--(\ref{epsilon(1,-1)})).

Continuing to follow \cite{L1}, we define the function 
\begin{equation} \label{eta}
\psi: \mathbb{Z}/4 \mathbb{Z} \times L \longrightarrow \langle i \rangle
\end{equation}
by the condition
\begin{equation} \label{eta-condition}
\hat{\nu}^p \iota(e_{\alpha}) = \psi (p, \alpha) \iota(e_{\nu^p 
  \alpha}),
\end{equation}
where we are using our particular choices of $\hat{\nu}$ (now extended
to $\C \{ L \}$) and of the section $e$ (and hence of the $2$-cocycle
$\epsilon_{C_0}$); recall that $\hat{\nu}^4=1$. Using
(\ref{particular-nu}) and (\ref{nu-alpha})--(\ref{-nu}) we have
\begin{equation} \label{eta-alpha} 
\psi (0, \alpha) =1, \; \; \psi (1, \alpha)=i, \; \; \psi (2,
 \alpha)= -1, \; \; \psi (3, \alpha)=-i \; \; \mbox{for} \; \; 
\alpha \in \{ \alpha_1, \alpha_2 \},
\end{equation}
\begin{equation} \label{eta--alpha} 
\psi (0, -\alpha) =1, \; \; \psi (1, -\alpha)=-i, \; \; \psi (2,
 -\alpha)= -1, \; \; \psi (3, -\alpha)=i \; \; \mbox{for} \; \; 
\alpha \in \{ \alpha_1, \alpha_2 \},
\end{equation}
and
\begin{equation} \label{eta-sum-alpha} \psi (p, \alpha) =1
  \; \; \mbox{for} \; \; \alpha = \pm (\alpha_1+\alpha_2), \; \; 
0 \leq p \leq 3.
\end{equation}

We extend the linear automorphism $\nu$ of $\frak{h}$ to a linear
automorphism, which we call $\hat{\nu}$, of $\frak{g}$, as
follows:
\begin{equation} \label{nu-g-1}
\hat{\nu} x_{\alpha}= \psi (1, \alpha) x_{\nu \alpha}
\end{equation}
for $\alpha \in \Delta$.  Then
\begin{equation} \label{nu-g}
\hat{\nu}^p x_{\alpha}= \psi (p, \alpha) x_{\nu^p \alpha}
\end{equation}
for $0 \leq p \leq 3$, 
$$
\hat{\nu}^4 =1 \; \; \; \mbox{on} \; \; \;  \frak{g},
$$ 
and $\hat{\nu}$
preserves $[ \cdot, \cdot ]$ and $\langle \cdot, \cdot \rangle$.  We
have
\begin{eqnarray} 
& \hat{\nu} x_{\alpha_1}= i x_{\alpha_2}, \; \; 
\hat{\nu} x_{\alpha_2}= i x_{\alpha_1}, \; \; \hat{\nu}  
x_{\alpha_1+\alpha_2}=x_{\alpha_1+\alpha_2}, &  \\ 
& \hat{\nu} x_{-\alpha_1}=-ix_{-\alpha_2}, \; \;
\hat{\nu} x_{-\alpha_2}=-ix_{-\alpha_1},\; \; 
\hat{\nu} x_{-\alpha_1-\alpha_2}=x_{-\alpha_1-\alpha_2}& \label{nu-g-} 
\end{eqnarray}
from (\ref{eta-alpha})--(\ref{eta-sum-alpha}).

For $n \in \Z$, set 
\begin{equation} \label{eigenspace}
\frak{g}_{(n)}= \{ x \in \frak{g} \; | \; \hat{\nu}(x)= i^{n} x \}, 
\end{equation}
the $i^{n}$-eigenspace for $\hat{\nu}$. Form the $\hat{\nu}$-twisted
affine Lie algebra associated to $\frak{g}$ and $\hat{\nu}$,
\begin{equation} \label{affine-twisted-g}
\hat{\frak{g}}[\hat{\nu}]=  \coprod _{n \in \mathbb{Z}} 
\frak{g}_{(n)} \otimes t^{\frac{n}{4}} \oplus \mathbb{C} {\bf k}
=\coprod _{n \in \frac{1}{4}\mathbb{Z}} 
\frak{g}_{(4n)} \otimes t^{n} \oplus \mathbb{C} {\bf k},
\end{equation}
where 
$$
[x\otimes t^m, y \otimes t^n]=[x,y] \otimes t^{m+n} +\langle x, y
\rangle  m \delta_{m+n,0} {\bf k}
$$
and
$$
[{\bf k}, \hat{\frak{g}}[\hat{\nu}]]=0
$$
for $m,n \in \frac{1}{4} \mathbb{Z}$ and $x \in \frak{g}_{(4m)}$ and
$y \in \frak{g}_{(4n)}$ (cf. \cite{K}). This is a copy of the twisted
affine Lie algebra $A_2^{(2)}$ and it is $\frac{1}{4} \Z$-graded. Set
\begin{equation} \label{largeaffine-twisted-g}
\tilde{\frak{g}}[\hat{\nu}]=\hat{\frak{g}}[\hat{\nu}] \oplus \mathbb{C}d,
\end{equation}
where the action of $d$ is given by $[d, x \otimes t^n]=nx\otimes t^n$
for $x \in \frak{g}$ and $n \in (1/4)\mathbb{Z}$ and $[d, {\bf k}]=0$.
 
The next theorem gives the $\hat{\frak{g}}[\hat{\nu}]$-module
structure of $V_L^T$:

\begin{theorem} \label{representation-1} (Theorem 9.1 of \cite{L1};
  see also Theorem 3 of \cite{FLM2}) The representation of
  $\hat{\frak{h}}[\nu]$ on $V_L^T$ extends uniquely to a Lie algebra
  representation of $\hat{\frak{g}}[\hat{\nu}]$ on $V_L^T$ such that
$$
(x_{\alpha})_{(4n)} \otimes t^n \mapsto x_{\alpha}^{\hat{\nu}}(n)
$$
for all  $n \in \frac{1}{4} \mathbb{Z}$ and $\alpha \in L$. Moreover,
  $V_L^T$ is an irreducible $\hat{\frak{g}}[\hat{\nu}]$-module.
\end{theorem} 

Throughout the rest of this paper, for $x \in \frak{g}$ and $n \in
\frac{1}{4} \Z$ we will write $x^{\hat{\nu}}(n)$ for the action of
$x_{(4n)} \otimes t^n \in \hat{\frak{g}}[\hat{\nu}]$ on any
$\hat{\frak{g}}[\hat{\nu}]$-module. In particular, we have the
operators $x_{\alpha}^{\hat{\nu}}(n)$ for $\alpha \in L$.  Also,
sometimes we write $x^{\hat{\nu}}(n)$ for the Lie algebra element $x_{(4n)}
\otimes t^n$, and it will be clear from the context whether
$x^{\hat{\nu}}(n)$ is an operator or a Lie algebra element.

The space $V_L^T$ is endowed with the tensor product weight
grading. As we mentioned in the previous section, this grading is
given by the action of $L^{\hat{\nu}}(0)$. By (\ref{grading-1}) and
(\ref{grading-2}) we have
\begin{equation} \label{weight-1-V_L^T}
L^{\hat{\nu}}(0) 1= \frac{1}{16} 1,
\end{equation}
and we write 
\[
\mbox{wt} \ 1 =\frac{1}{16},
\]
where by $1$ we mean $1 \otimes 1 \in S[\nu] \otimes U_T$ (recall
(\ref{T}) and (\ref{U_T-A_2^2})), and by using (\ref{grading-4}) we
have
\begin{equation}
L^{\hat{\nu}}(0) v = \left(\mbox{wt} \; v +\frac{1}{16} \right) v
\end{equation}
for a homogeneous element $v \in V_L^T$.

Taking $a=e_{\alpha}$ in (\ref{important}) we obtain
\begin{equation} \label{important-A_2^2}
[L^{\hat{\nu}}(0), x^{\hat{\nu}}_{\alpha}(n)]=
\left(-n-1+\frac{1}{2} \langle \alpha, \alpha \rangle \right) 
x_{\alpha}^{\hat{\nu}}(n),
\end{equation}
and thus
\begin{equation} \label{weight_grading_x} \mbox{wt}
 \  x_{\alpha}^{\hat{\nu}}(n)=-n-1+\frac{1}{2} \langle \alpha, \alpha
  \rangle
\end{equation}
for any $n \in (1/4) \Z$.

The space $V_L^T$ has also a {\it charge} grading given by the
eigenvalues of the operator $(\alpha_1+\alpha_2) (0)$, and this
grading is compatible with the weight grading. Thus
$x^{\hat{\nu}}_{\alpha}(n)$ has charge $\langle \alpha_1+\alpha_2,
\alpha \rangle$ for any $n \in (1/4) \Z$, where
$x_{\alpha}^{\hat{\nu}}(n)$ is viewed again as either an operator or
as an element of $U(\bar{\frak{n}}[\hat{\nu}])$. 

Now as a consequence of (\ref{Y-hat}) and (\ref{Y-hat-r}) (recall that
in our specialized setting, $k=4$ and $\eta=i$), we obtain the
following linear relations among the operators
$x^{\hat{\nu}}_{\alpha}(m)$ for $\alpha \in \{\alpha_1, \alpha_2,
\alpha_1+\alpha_2\}$ and $m \in \frac{1}{4} \mathbb{Z}$:

\begin{lem} \label{lemma1}
We have
\begin{equation} \label{alpha-0}
x_{\alpha}^{\hat{\nu}}(m)=0 \; \; \mbox{if} \; \; \alpha \in 
\{ \alpha_1, \alpha_2 \} \; \; \mbox{and} \; \; m \in \frac{1}{2} 
\mathbb{Z},
\end{equation}
\begin{equation} \label{1=2}
x_{\alpha_2}^{\hat{\nu}}(m)= x_{\alpha_1}^{\hat{\nu}}(m) \; \; 
\mbox{if} \; \; m \in
\frac{1}{4} + \mathbb{Z},
\end{equation}
\begin{equation} \label{1=-2} 
x_{\alpha_2}^{\hat{\nu}}(m)=-x_{\alpha_1}^{\hat{\nu}}(m) \; \; 
\mbox{if} \; \; m \in
\frac{3}{4} + \mathbb{Z}
\end{equation}
and
\begin{equation} \label{1-2} 
x_{\alpha_1+\alpha_2}^{\hat{\nu}}(m)=0 \; \; \mbox{if} \; \; 
m \in \frac{1}{4} 
\mathbb{Z} \setminus \mathbb{Z}.
\end{equation}
\end{lem}

\begin{proof} We first show that (\ref{alpha-0}) holds for
$\alpha_1$. By taking $r=2$ and $v=\iota(e_{\alpha_1})$ in
(\ref{Y-hat-r}) and by using (\ref{eta-condition}) for $p=2$ and
$\alpha=\alpha_1$ we obtain
\begin{equation} \nonumber
Y^{\hat{\nu}}(-\iota(e_{\alpha_1}), x)=\lim_{x^{1/4} \rightarrow 
-x^{1/4}} Y^{\hat{\nu}}(\iota(e_{\alpha_1}), x).
\end{equation}
This immediately yields
$$
-\sum_{m \in (1/4)\Z} x^{\hat{\nu}}_{\alpha_1}(m)x^{-m-1}= \sum_{(1/4)\Z} 
(-1)^{-4m} x^{\hat{\nu}}_{\alpha_1}(m) x^{-m-1},
$$
and thus (\ref{alpha-0}) for $\alpha_1$. The proof of the same formula
for $\alpha_2$ instead of $\alpha_1$ is completely analogous.

Next we use (\ref{Y-hat}) for $v=\iota(e_{\alpha_1})$ together with
(\ref{eta-condition}) for $p=1$ and $\alpha=\alpha_1$ and thus we get
\begin{equation} \nonumber
iY^{\hat{\nu}}(\iota(e_{\alpha_2}), x)=\lim_{x^{1/4} \rightarrow 
i^{-1}x^{1/4}} Y^{\hat{\nu}}(\iota(e_{\alpha_1}), x).
\end{equation}
This implies
$$
\sum_{m \in (1/4)\Z} ix^{\hat{\nu}}_{\alpha_2}(m)x^{-m-1}= \sum_{(1/4)\Z} 
i^{4m} x^{\hat{\nu}}_{\alpha_1}(m) x^{-m-1},
$$
which gives (\ref{1=2}) and (\ref{1=-2}).

By using (\ref{Y-hat}) for $v=\iota(e_{\alpha_1+\alpha_2})$ together
with (\ref{eta-condition}) for $p=1$ and $\alpha=\alpha_1+\alpha_2$ we
obtain
$$
\sum_{m \in (1/4)\Z} x^{\hat{\nu}}_{\alpha_1+\alpha_2}(m)x^{-m-1}
=\sum_{m \in (1/4)\Z} i^{4m}
x^{\hat{\nu}}_{\alpha_1+\alpha_2}(m)x^{-m-1},
$$
which implies formula (\ref{1-2}).
\end{proof}
\vspace{1em}

Note that by the above lemma and (\ref{operators-x}) we have
\begin{equation}
Y^{\hat{\nu}}(\iota(e_{\alpha}), x)= 
\sum_{m \in \frac{1}{4}+\frac{1}{2} \mathbb{Z}}
x_{\alpha}^{\hat{\nu}}(m)x^{-m-1} \; \; \mbox{for} \; \; \alpha \in 
\{ \alpha_1, \alpha_2 \}
\end{equation}
and
\begin{equation}
Y^{\hat{\nu}}(\iota(e_{\alpha_1+\alpha_2}), x)= 
\sum_{m \in \mathbb{Z}}
x_{\alpha_1+\alpha_2}^{\hat{\nu}}(m)x^{-m-1}.
\end{equation} 

By using the commutator formula (\ref{commutator}) for twisted vertex
operators we obtain the following brackets:

\begin{lem}\label{brackets} For $m, n \in \frac{1}{4} +\frac{1}{2} \Z$
  we have
  \begin{equation}\label{bracket1} [x_{\alpha_1}^{\hat{\nu}}(m),
    x_{\alpha_2}^{\hat{\nu}}(n)]=\frac{1}{2}
    x_{\alpha_1+\alpha_2}^{\hat{\nu}}(m+n),
\end{equation}
\begin{equation}\label{bracket2} [x_{\alpha_1}^{\hat{\nu}}(m),
  x_{\alpha_1}^{\hat{\nu}}(n)]=-\frac{i}{4}
  (i^{-4m}-(-i)^{-4m})x_{\alpha_1+\alpha_2}^{\hat{\nu}}(m+n),
\end{equation}
\begin{equation}\label{bracket3} [x_{\alpha_2}^{\hat{\nu}}(m),
  x_{\alpha_2}^{\hat{\nu}}(n)]=\frac{i}{4}
  (i^{-4m}-(-i)^{-4m})x_{\alpha_1+\alpha_2}^{\hat{\nu}}(m+n),
\end{equation}
and for $m \in \Z$, $n \in \frac{1}{4}\Z$ and $\alpha \in \{ \alpha_1,
\alpha_2, \alpha_1+\alpha_2 \}$ we have 
\begin{equation}\label{bracket4}[x_{\alpha_1+\alpha_2}^{\hat{\nu}}(m),
  x_{\alpha}^{\hat{\nu}}(n)]=0.
\end{equation}
\end{lem}

For $n \in \Z$ we shall identify $\frak{g}_{(n \; \rm{mod} \; 4)}$
with $\frak{g}_{(n)}$, the $i^n$-eigenspace of $\nu$ in $\frak{g}$
(recall (\ref{eigenspace})). By (\ref{twist}) and (\ref{nu-g-1}) we
have
\begin{equation}
\goth{g}_{(0)} = \mathbb{C}x_{\alpha_1+\alpha_2}
\oplus \mathbb{C} (\alpha_1+ \alpha_2) \oplus
\mathbb{C}x_{-\alpha_1-\alpha_2} =\goth{g}_{(4m)},
\end{equation}
\begin{equation}
\goth{g}_{(1)}=\mathbb{C}(x_{\alpha_1}+x_{\alpha_2})
\oplus \mathbb{C} (x_{-\alpha_1}-x_{-\alpha_2}) =\goth{g}_{(4m+1)},
\end{equation}
\begin{equation}
\goth{g}_{(2)}= \mathbb{C} (\alpha_1-\alpha_2) = \goth{g}_{(4m+2)},
\end{equation}
\begin{equation}
\goth{g}_{(3)}= \mathbb{C}
(x_{\alpha_1}-x_{\alpha_2}) \oplus \mathbb{C} (x_{-\alpha_1}+
x_{-\alpha_2})
=\goth{g}_{(4m+3)}
\end{equation}
for any $m \in \mathbb{Z}$, and
$$
\frak{g}=\coprod_{p \in \mathbb{Z}/4\Z} \frak{g}_{(p)}.
$$
Note that the twisted affine Lie algebra (\ref{affine-twisted-g})
corresponding to $\frak{g}$ and to the automorphism $\hat{\nu}$ of
$\frak{g}$ decomposes as
\begin{eqnarray}   \nonumber
\hat{\goth{g}}[\hat{\nu}] & = & \goth{g}_{(0)} \otimes 
\mathbb{C}[t, t^{-1}] \oplus \goth{g}_{(1)} \otimes t^{1/4} 
\mathbb{C} [t, t^{-1}]  \oplus \\ \nonumber
&& \goth{g}_{(2)} \otimes t^{1/2} \mathbb{C}[t, t^{-1}] 
\oplus \goth{g}_{(3)} \otimes t^{3/4} \mathbb{C}[t, t^{-1}] 
\oplus \mathbb{C}{\bf k}. \\ \nonumber
\end{eqnarray}

For any $\hat{\nu}$-stable Lie subalgebra $\frak{u}$ of $\frak{g}$, we
shall write $\hat{\frak{u}}[\hat{\nu}]$ for the correspondingly
defined twisted affine Lie subalgebra of
$\hat{\frak{g}}[\hat{\nu}]$. In particular, we may write
$\hat{\frak{h}}[\nu]$ (recall (\ref{twisted})) as
$\hat{\frak{h}}[\hat{\nu}]$. Take the $\hat{\nu}$-stable Lie
subalgebra
\begin{equation} \label{n}
\frak{n}= \C x_{\alpha_1} \oplus \C x_{\alpha_2} \oplus \C x_{\alpha_1+\alpha_2}
\end{equation}
of $\frak{g}$, and consider the
twisted affinization of $\frak{n}$,
\begin{equation} \label{affinization-n}
\hat{\frak{n}}[\hat{\nu}]= \coprod_{r \in \Z} \frak{n}_{(r)} 
\otimes t^{\frac{r}{4}} \oplus \C {\bf k} = 
\coprod_{r \in \frac{1}{4} \Z} \frak{n}_{(4r)} 
\otimes t^r  
\oplus \C {\bf k} \subset \hat{\frak{g}}[\hat{\nu}],
\end{equation}
where for $r \in \Z$, $\frak{n}_{(r)}$ is the $i^r$-eigenspace of
$\frak{n}$ for $\hat{\nu}$ (as in (\ref{eigenspace})). As in
\cite{CalLM1}--\cite{CalLM3}, we drop the $1$-dimensional space $\C
{\bf k}$ in (\ref{affinization-n}) and use instead the subalgebras
\begin{equation} \label{affine-n}
\bar{\frak{n}}[\hat{\nu}]= \coprod_{r \in \Z} \frak{n}_{(r)} 
\otimes t^{\frac{r}{4}} = \coprod_{r \in \frac{1}{4} \Z} \frak{n}_{(4r)} 
\otimes t^r,
\end{equation}
\begin{equation} \label{affine-n+} \bar{\frak{n}}[\hat{\nu}]_{+}=
  \coprod_{r \geq 0} \frak{n}_{(r)} \otimes t^{\frac{r}{4}}
  =\coprod_{r \in \frac{1}{4} \Z, r \geq 0} \frak{n}_{(4r)} \otimes
  t^r
\end{equation}
and
\begin{equation} \label{affine-n-} \bar{\frak{n}}[\hat{\nu}]_{-}=
  \coprod_{r<0} \frak{n}_{(r)} \otimes t^{\frac{r}{4}} =\coprod_{r \in
    \frac{1}{4}, r<0} \frak{n}_{(4r)} \otimes t^r
\end{equation}
of $\hat{\frak{n}}[\hat{\nu}]$ (note that the form $\langle \cdot,
\cdot \rangle$ vanishes on $\frak{n}$).

\section{Shifted operators}

\setcounter{equation}{0} In this section we shift the
$\hat{\nu}$-twisted vertex operators (\ref{L-operator}) and
(\ref{Ynuhat}) by using an element of $\frak{h}_{(0)}$ that lies in
the rational span of the lattice $L$, and we recall a shifted twisted
vertex operator construction of $A_2^{(2)}$. We follow Section 10 of
\cite{L1}.

Fix 
\[
\gamma \in \frak{h}_{(0)}.
\]
Let $a,b \in \Lnu$, $u \in T$ and $h
\in \h_{(0)}$.  Define a $\gamma$-shifted action of $\frak{h}_{(0)}$
on $U_T$ by
\begin{equation}\label{shifted-action-h}
  h^{\gamma} \cdot b \otimes u =  \langle h, \overline{b}+\gamma \rangle b 
  \otimes u
\end{equation}
and the $\mbox{End}\, U_T$-valued formal Laurent series
$x^{h^{\gamma}}$ by
\begin{equation}
  x^{h^{\gamma}} \cdot b \otimes u=
  x^{\langle h, \overline{b}+\gamma \rangle} b \otimes u.
\end{equation}
Then we have 
\begin{equation}
  x^{h^{\gamma}} a= ax^{\langle h, \overline{a}+\gamma \rangle + h} \ \; 
  \mbox{for} \; \; h \in
  \h_{(0)}.
\end{equation}
Recall (\ref{action-h}), (\ref{x^h})  and (\ref{x-a}).

Now 
$$
U_T =\coprod_{\alpha \in P_0L} U_{\alpha},
$$
where 
$$
U_{\alpha}= \{ u \in U_T \; | \; 
h^{\gamma} \cdot u = \langle h, \alpha + \gamma \rangle u 
\; \; \mbox{for} \; \; h \in \frak{h}_{(0)} \}.
$$ 
The space $U_T$, and hence $V_L^T$, has a new
$\hat{\h}[\nu]$-structure, with $\hat{\h}[\nu]_{\frac{1}{k}\Z}$ acting
trivially and $\h_{(0)}$ as in (\ref{shifted-action-h}).

Continuing to follow \cite{L1} we define the $\gamma$-shifted
$\hat{\nu}$-twisted vertex operator on $V_L^T$ for $a \in \hat{L}$, as follows:
\begin{equation} \label{shifted-L-operator} 
Y^{\hat{\nu}, \gamma}(\iota(a), x)= k^{-\langle \overline{a},
\overline{a}\rangle /2} 
\sigma(\overline{a}) E^{-} (-\overline{a}, x) E^{+} (-\overline{a}, x)
a x^{\overline{a}_{(0)}^{\gamma}+\langle
\overline{a}_{(0)},
\overline{a}_{(0)}\rangle /2-\langle \overline{a},
\overline{a}\rangle /2}.
\end{equation}
Note that 
\begin{equation} 
  Y^{\hat{\nu}, \gamma}(\iota(a),
  x)=Y^{\hat{\nu}}(\iota(a), x)x^{\langle {\overline{a}}_{(0)}, \gamma \rangle}.
\end{equation}
Define the component operators $x_{\alpha}^{\hat{\nu}, \gamma}(n)$ for
$n \in \frac{1}{k}\Z$ and $\alpha \in L$ by
\begin{equation}
  Y^{\hat{\nu}, \gamma}(\iota(e_{\alpha}), x)= \sum _{n \in 
    \frac{1}{k}\Z}x_{\alpha}^{\hat{\nu}, \gamma}(n)
  x^{-n-\frac{\langle \alpha, \alpha \rangle}{2}}.
\end{equation}
Then  for $n \in \frac{1}{k}\Z$ we have 
\begin{equation} \label{shift}
x_{\alpha}^{\hat{\nu},
  \gamma}(n)=x_{\alpha}^{\hat{\nu}}(n+\langle \alpha_{(0)}, \gamma
\rangle),
\end{equation}
as operators on $V_L^T$. 

For $v=\alpha_1(-n_1) \cdots \alpha_m(-n_m) \cdot \iota(a) \in V_L$ with $\alpha_1,
\dots, \alpha_m \in \h$, $n_1,\dots, n_m >0$ and $a \in \hat{L}$, one can define the shifted
version of the $\hat{\nu}$-twisted vertex operators (\ref{Ynuhat}):
\begin{equation} \label{shifted-general}
Y^{\hat{\nu}, \gamma}(v, x)=W^{\gamma}(e^{\Delta_ x}v,x),
\end{equation}
where 
\begin{equation}
W^{\gamma} (v, x)= \ _\circ^\circ
\biggl
(\frac{1}{(n_1-1)!}\left(\frac{d}{dx}\right)^{n_1-1}\alpha_1^{\hat{\nu}}(x)\biggr)\cdots
\left(\frac{1}{(n_m-1)!}\biggl(\frac{d}{dx}\right)^{n_m-1}
\alpha_m^{\hat{\nu}}(x)\biggr)Y^{\hat{\nu}, \gamma}(\iota(a),x) \ _\circ^\circ.
\end{equation}

\begin{remark} 
  \rm The operators (\ref{shifted-L-operator}) correspond to
  the shifted operators $Y^{\gamma}(a, \zeta)$ in \cite{L1}. By
  replacing $\zeta$ by $x^{-\frac{1}{k}}$ in $Y^{\gamma}(a, \zeta)$
  and then by multiplying by $x^{-\frac{\langle \overline{a},
      \overline{a} \rangle }{2}}$ one obtains the $\gamma$-shifted
  $\hat{\nu}$-twsited vertex operator $Y^{\hat{\nu}, \gamma}(\iota(a),
  x)$. Here we identify $a$ with $\iota(a)$.
\end{remark}

Now suppose that $\gamma$ lies in the rational span of $L$. Choose ${\bar
  k} \in \Z_{+}$ such that
\begin{equation} \label{gamma-condition}
{\bar k} \langle \alpha, \gamma \rangle \in \frac{1}{k} \Z
\end{equation}
for $\alpha \in L$. 

For $a \in L$ we define the operator $Y^{\hat{\nu}, \gamma, {\bar
    k}}(\iota(a), x)$ as follows:
\begin{equation}
  Y^{\hat{\nu}, \gamma, {\bar k}}(\iota(a), x)=
Y^{\hat{\nu}, \gamma}(\iota(a), x^{\bar k}).
\end{equation}
For $\alpha \in L$ and $n \in \frac{1}{k \bar{k}}\Z$ define the
operators $x_{\alpha}^{\hat{\nu}, \gamma, {\bar k}}(n)$ by
\begin{equation} \label{new-operators} Y^{\hat{\nu}, \gamma, {\bar
      k}}(\iota(e_{\alpha}), x)=\sum_{n \in \frac{1}{k {\bar k}}\Z
  }x_{\alpha}^{\hat{\nu}, \gamma, {\bar k}}(n)x^{-n- \frac{\langle
      \alpha, \alpha \rangle}{2}}.
\end{equation}
 
For the rest of this section we specialize $L$ to the root lattice of
$\frak{sl}(3, \C)$ and $\nu$ as in (\ref{twist}). From the previous
section, $k=4$ and $\eta =i$. Let $\eta_{\gamma} \in \mathbb{C}$ be a
primitive $4{\bar k}$th root of unity such that
\begin{equation}
\eta_{\gamma}^{\bar k}=i.
\end{equation}
Define
\begin{equation} \label{eta-gamma} \psi_{\gamma}: \mathbb{Z}/4{\bar k}
  \mathbb{Z} \times L \longrightarrow \langle \eta_{\gamma} \rangle
\end{equation}
by
\begin{equation}
  \psi_{\gamma}(p, \alpha)=
\eta^{-4{\bar k}p \langle \alpha, \gamma \rangle}_{\gamma} \psi(p, \alpha)
\end{equation}
(recall (\ref{eta})).  Denote by $\hat{\nu}_{\gamma}$ the shifted
version of (\ref{nu-g-1}):
 \begin{equation}\label{nu-g-1-gamma}
\hat{\nu}_{\gamma} h=\nu h \; \; \; \mbox{for} \; \; \; h \in \frak{h}
\end{equation}
and 
\begin{equation}
  \hat{\nu}_{\gamma}x_{\alpha}=
  \eta_{\gamma}^{-4 {\bar k}\langle \alpha, \gamma \rangle} \hat{\nu} 
x_{\alpha} \; \; \; 
  \mbox{for} \; \; \; \alpha \in L,
\end{equation}
where $\hat{\nu}$ is as in (\ref{nu-g-1}).  Then $\hat{\nu}_{\gamma}$ is a
Lie algebra automorphism of $\frak{g}$ such that
\begin{equation}
  \hat{\nu}_{\gamma} x_{\alpha}=\psi_{\gamma}(1, \alpha)x_{\nu \alpha} \; \; \; 
\mbox{for} \; \; \; \alpha \in \Delta
\end{equation}
and
\begin{equation}
\hat{\nu}_{\gamma}^{4{\bar k}}=1 \; \; \; \mbox{on} \; \; \; \frak{g}.
\end{equation}

For $n \in \Z$ set
\begin{equation}
  \frak{h}_{(n)}= \{ 
h \in \frak{h} \; | \; \hat{\nu}_{\gamma} h=\eta^n_{\gamma} h \}.
\end{equation}
Consider the $\hat{\nu}_{\gamma}$-twisted affine Lie algebra
associated with $\frak{h}$ and $\hat{\nu}_{\gamma}$:
\begin{equation} \label{gamma-twisted-h}
\hat{\frak{h}}[\hat{\nu}_{\gamma}] = \coprod _{n \in \mathbb{Z}} 
\frak{h}_{(n)} \otimes t^{\frac{n}{4{\bar k}}} \oplus \mathbb{C} {\bf k}
=\coprod _{n \in \frac{1}{4{\bar k}}\mathbb{Z}} 
\frak{h}_{(4{\bar k}n)} \otimes t^{n} \oplus \mathbb{C} {\bf k}
\end{equation}
with the usual brackets. Since the Lie algebras $\hat{\frak
  h}[\hat{\nu}_{\gamma}]$ and $\hat{\frak h}[\nu]$ are isomorphic,
$V_L^T$ is an $\hat{\frak h}[\hat{\nu}_{\gamma}]$-module.  Now
consider $\hat{\frak g}[\hat{\nu}_{\gamma}]$, the
$\hat{\nu}_{\gamma}$-twisted affine Lie algebra associated to $\frak
g$ and $\hat{\nu}_{\gamma}$:
\begin{equation} \label{gamma-twisted-g}
\hat{\frak{g}}[\hat{\nu}_{\gamma}] = \coprod _{n \in \mathbb{Z}} 
\frak{g}_{(n)} \otimes t^{\frac{n}{4{\bar k}}} \oplus \mathbb{C} {\bf k}
=\coprod _{n \in \frac{1}{4{\bar k}}\mathbb{Z}} 
\frak{g}_{(4{\bar k}n)} \otimes t^{n} \oplus \mathbb{C} {\bf k},
\end{equation}
where
\begin{equation}
  \frak{g}_{(n)}= 
\{ x \in \frak{g} \; | \; \hat{\nu}_{\gamma} x=\eta^n_{\gamma} x \}
\end{equation}
for $n \in \Z$, with the corresponding brackets. This is a copy of
$A_2^{(2)}$. Recall (\ref{new-operators}), the component operators of
the vertex operator $Y^{\hat{\nu}, \gamma, {\bar
    k}}(\iota(e_{\alpha}), x)$. The analogue of Theorem
\ref{representation-1} holds:
\begin{theorem} \label{representation-2} (Theorem 10.1 of \cite{L1})
  The representation of $\hat{\frak{h}}[\nu_{\gamma}]$ on $V_L^T$
  extends uniquely to a Lie algebra representation of
  $\hat{\frak{g}}[\hat{\nu}_{\gamma}]$ on $V_L^T$ such that
$$
(x_{\alpha})_{(4{\bar k}n)} \otimes t^n \mapsto x_{\alpha}^{\hat{\nu},
  \gamma, {\bar k}}(n)
$$
for all $n \in \frac{1}{4{\bar k}} \mathbb{Z}$ and $\alpha \in L$, and
this representation is irreducible.
\end{theorem}

Now assume that $\bar{k}=1$ and $\gamma$ is an element of
$\frak{h}_{(0)}$ that lies in the rational span of $L$ and such that
it satisfies (\ref{gamma-condition}). From now on we will denote by
$V_L^{T, \gamma}$ the $\hat{\frak{g}}[\hat{\nu}_{\gamma}]$-module
$V_L^T$. It was proved in \cite {Li1} (see also Proposition 2.14 in
\cite{DLM}) that $(V_L^{T, \gamma}, Y^{\hat{\nu}, \gamma})$ is a
$\sigma_{\gamma} \hat{\nu}$-twisted module for $V_L$, where
\begin{equation}
\sigma_{\gamma}=e^{-2 \pi i \gamma},
\end{equation}
an automorphism of $V_L$, and the vertex operator $Y^{\hat{\nu},
  \gamma}$ is as in (\ref{shifted-L-operator}) and
(\ref{shifted-general}). Recall the component operators
$x_{\alpha}^{\hat{\nu}, \gamma}(n)$ for $\alpha \in L$, $n \in
\frac{1}{k} \Z$, and (\ref{shift}).

Similarly, $(V_L^{\gamma}, Y^{\gamma})$ is a $\sigma_{\gamma}$-twisted module
for $V_L$, where $V^{\gamma}_L$ is the vector space $V_L$, and for $a
\in \hat{L}$,
\begin{equation}  \label{shifted=operator-V}
Y^{\gamma}(\iota(a), x)=E^{-}(-\bar{a}, x)E^{+}(-\bar{a}, x)a
x^{{\bar a}^{\gamma}},
\end{equation}
and, more generally, for $v = \alpha_1(-n_1)\cdots
\alpha_m(-n_m) \otimes \iota(a) \in V_L\nonumber $ with $\alpha_1,
\dots, \alpha_m \in \h$, $n_1,\dots, n_m >0$ and $a \in \hat{L}$, we set
\begin{equation} \label{shifted-operator}
Y^{\gamma}(v,x) = \ _\circ^\circ \left(\frac{1}{(n_1-1)!}\left(\frac{
d}{dx}\right)^{n_1-1} \alpha_1(x)\right)\cdots
\left(\frac{1}{(n_m-1)!}
\left(\frac{d}{dx}\right)^{n_m-1}\alpha_m(x)\right)Y^{\gamma}(\iota(a),x) \
_\circ^\circ.
\end{equation}
Define the component operators $x_{\alpha}^{\gamma}(n)$ for $\alpha
\in L$ and $n \in (1/k)\mathbb{Z}$:
\[
Y^{\gamma}(\iota(e_{\alpha}), x) =\sum_{n \in (1/k)\mathbb{Z}}
x_{\alpha}^{\gamma}(n)x^{-n - \frac{\langle \alpha, \alpha
    \rangle}{2}}.
\] 
Note that
\[
Y^{\gamma}(\iota(e_{\alpha}), x)=Y(\iota(e_{\alpha}), x)x^{\langle
  \alpha, \gamma \rangle}.
\]
By Proposition 2.15 in \cite{DLM} we have that the
$\sigma_{\gamma}$-twisted $V_L$ module $(V_L^{\gamma}, Y^{\gamma})$
is naturally isomorphic to the $\sigma_{\gamma}$-twisted $V_L$ module
$V_{L+\gamma}$.  

\section{Principal subspaces of standard $A_2^{(2)}$-modules}
\setcounter{equation}{0}

Recall the twisted affine Lie algebras $\hat{\frak{g}}[\hat{\nu}]$,
$\tilde{\frak{g}}[\hat{\nu}]$ and their subalgebras
$\bar{\frak{n}}[\hat{\nu}]$, $\bar{\frak{n}}[\hat{\nu}]_{+}$ and
$\bar{\frak{n}}[\hat{\nu}]_{-}$ (see
(\ref{affine-twisted-g})--(\ref{largeaffine-twisted-g}) and
(\ref{affine-n})--(\ref{affine-n-})). The notion of principal
subspace of a highest weight module for an affine Lie algebra was
introduced in \cite{CalLM2} as a straightforward generalization of the
principal subspace of a standard module for an untwisted affine Lie
algebra of type $A$ as in \cite{FS1}--\cite{FS2}.

\begin{defn}
  For any standard $\tilde{\frak{g}}[\hat{\nu}]$-module $V$ we define
  its principal subspace $W$ to be
\[
W=U(\bar{\frak{n}}[\hat{\nu}]) \cdot v,
\]
where $v$ is a highest weight vector of $V$. 
\end{defn}

\begin{remark} {\em This definition can be used to define
    principal subspaces for an arbitrary twisted affine
    Lie algebra $\hat{\frak{g}}[\hat{\nu}]$,
    whenever the automorphism playing the role of $\hat{\nu}$
    preserves the subalgebra generalizing $\frak{n}$, and not only
    $A_2^{(2)}$. If necessary, we can also relax the condition that $v$ be a
    highest weight vector and instead allow $v$ to be a more general
    vector (see \cite{CalLM3}).
    }
\end{remark}

In particular,  we have
the principal subspace, denoted by $W_L^T$, of $V_L^T$,
\begin{equation} \label{ps}
W_L^T= U(\bar{\frak{n}}[\hat{\nu}]) \cdot v_{\Lambda},
\end{equation}
where
\[
\Lambda \in (\frak{h}_{(0)} \oplus \C {\bf k} \oplus \mathbb{C}d)^*
\]
is the fundamental weight of $\tilde{\frak{g}}[\hat{\nu}]$ defined by
$\langle \Lambda , {\bf k} \rangle =1$, $\langle \Lambda,
\frak{h}_{(0)} \rangle = 0$ and $\langle \Lambda, d \rangle =0$, and
$v_{\Lambda}$ is a highest weight vector of $V_L^T$. Then
\begin{equation} \label{ps-}
W_L^T= U(\bar{\frak{n}}[\hat{\nu}]_{-}) \cdot v_{\Lambda}.
\end{equation}
We shall take our highest weight vector $v_{\Lambda}$ of $V_L^T$ to be
\begin{equation} \label{hwv}
v_{\Lambda}=1 \in V_L^T,
\end{equation}
where by $1$ we mean $1 \otimes 1 \in S[\nu] \otimes U_T$ (recall
(\ref{T}) and (\ref{U_T-A_2^2})).

Consider the surjective map
\begin{eqnarray} 
F_{\Lambda}: U(\hat{\frak{g}}[\hat{\nu}]) & \longrightarrow & V_L^T 
\\
a & \mapsto & a \cdot v_{\Lambda} \nonumber
\end{eqnarray}
and denote by $f_{\Lambda}$ the restriction of $F_{\Lambda}$ to
$U(\bar{\frak{n}}[\hat{\nu}])$:
\begin{eqnarray}  \label{f-lambda}
f_{\Lambda}: U(\bar{\frak{n}}[\hat{\nu}]) 
& \longrightarrow & W_L^T \\
a & \mapsto & a \cdot v_{\Lambda}. \nonumber
\end{eqnarray}
One of the main goals of this paper is to give a precise description
of the kernel $\mbox{Ker} \; f_{\Lambda}$, and thus a presentation of
the principal subspace $W_L^T$.

As in \cite{CalLM1}--\cite{CalLM3} we will use principal subspaces of
generalized Verma modules. Set
\begin{equation}
  N_L^T= U(\hat{\frak{g}}[\hat{\nu}])  
\otimes_{U(\hat{\frak{g}}[\hat{\nu}])_{\geq 0}} \mathbb{C}v_{\Lambda}^N,
\end{equation}
where 
\[
\hat{\frak{g}}([\hat{\nu}])_{\geq 0}= 
\left ( \coprod_{n \geq 0} \frak{g}_{(n)} \otimes t^{n/4} \right ) 
\oplus \mathbb{C} {\bf k}.
\]
Define the principal subspace of the generalized Verma module $N_L^T$,
\begin{equation}
  W_L^{T, N}= U(\bar{\frak{n}}[\hat{\nu}]) \cdot v_{\Lambda}^N \subset N_L^T.
\end{equation}
We have the following surjective maps:
\begin{eqnarray} \nonumber
F_{\Lambda}^N: U(\hat{\frak{g}}[\hat{\nu}]) & \longrightarrow & N_L^T 
\nonumber \\
a & \mapsto & a \cdot v_{\Lambda}^N, \nonumber
\end{eqnarray}
\begin{eqnarray} \nonumber
f_{\Lambda}^N: U(\bar{\frak{n}}[\hat{\nu}]) 
& \longrightarrow & W_L^{T, N} \\
a & \mapsto & a \cdot v_{\Lambda}^N \nonumber
\end{eqnarray}
and
\begin{eqnarray} \nonumber
\Pi_{\Lambda}: N_L^T & \longrightarrow & V_L^T 
\nonumber \\
a \cdot v_{\Lambda}^N & \mapsto & a \cdot v_{\Lambda}, \nonumber
\end{eqnarray}
\begin{eqnarray} \nonumber
\pi_{\Lambda}: W_L^{T, N} & \longrightarrow & W_L^T 
\nonumber \\
a \cdot v_{\Lambda}^N & \mapsto & a \cdot v_{\Lambda}, \nonumber
\end{eqnarray}

\begin{theorem} \label{relations V_L^T} On the standard
  $\hat{\frak{g}}[\hat{\nu}]$-module $V_L^T$ we have:
\begin{equation}
\lim_{x_2^{1/4} \mapsto x_1^{1/4}}(x_1^{1/2}+x_2^{1/2}) 
Y^{\hat{\nu}}(\iota(e_{\alpha_j}), x_1)Y^{\hat{\nu}}
(\iota(e_{\alpha_j}), x_2)=0 \; \; \mbox{for} \; \; j=1,2,
\end{equation}
\begin{equation}
\lim_{x_2^{1/4} \mapsto ix_1^{1/4}}(x_1^{1/2}-x_2^{1/2})
Y^{\hat{\nu}}(\iota(e_{\alpha_1}), x_1)Y^{\hat{\nu}}(\iota(e_{\alpha_2}), x_2)=0,
\end{equation}
\begin{equation}
Y^{\hat{\nu}}(\iota(e_{\alpha_1+\alpha_2}), x)^2=0
\end{equation}
and 
\begin{equation}
Y^{\hat{\nu}}(\iota(e_{\alpha_j}), x)
Y^{\hat{\nu}}(\iota(e_{\alpha_1+\alpha_2}), x)=0 \; \; \mbox{for} \; \; j=1,2.
\end{equation}
\end{theorem}
\begin{proof} This follows by the use of (\ref{e-series}).
\end{proof}

By analogy with the corresponding constructions in
\cite{CalLM1}--\cite{CalLM3}, we introduce the following formal
infinite sums indexed by $t \in \frac{1}{4}\Z$:
\begin{equation} \label{R_j} R_{j;
    t}=\sum_{{\tiny \begin{array}{c}n_1, n_2 \in (1/4) + (1/2) \Z \\
        n_1+n_2+1/2=-t \end{array}}} \left (x_{\alpha_j}^{\hat{\nu}}
    \left (n_1+\frac{1}{2} \right ) x_{\alpha_j}^{\hat{\nu}}(n_2)
    +x_{\alpha_j}^{\hat{\nu}}(n_1)x_{\alpha_j}^{\hat{\nu}} \left
      (n_2+\frac{1}{2} \right ) \right ) \; \; \; \mbox{for} \; \;
  j=1,2,
\end{equation}
\begin{equation}\label{R_{1,2}}
  R_{1;2;t}=\sum_{{\tiny \begin{array}{c}n_1, n_2 \in (1/4) + (1/2) \Z \\
        n_1+n_2+1/2=-t \end{array}}} \left
    (x_{\alpha_1}^{\hat{\nu}} \left (n_1+\frac{1}{2} \right )
    x_{\alpha_2}^{\hat{\nu}}(n_2)
    -x_{\alpha_1}^{\hat{\nu}}(n_1)x_{\alpha_2}^{\hat{\nu}} \left
      (n_2+\frac{1}{2} \right ) \right ),
\end{equation}
\begin{equation}\label{R_{2,1}}
  R_{2;1;t}=\sum_{{\tiny \begin{array}{c}n_1, n_2 \in (1/4) + (1/2) \Z \\
        n_1+n_2+1/2=-t \end{array}}} \left
    (x_{\alpha_2}^{\hat{\nu}} \left (n_1+\frac{1}{2} \right )
    x_{\alpha_1}^{\hat{\nu}}(n_2)
    -x_{\alpha_2}^{\hat{\nu}}(n_1)x_{\alpha_1}^{\hat{\nu}} \left
      (n_2+\frac{1}{2} \right ) \right ),
\end{equation}
\begin{equation} \label{R_{12}}
  R_{1,2; t}=\sum_{{\tiny \begin{array}{c} m_1, m_2 \in \Z\\
        m_1+m_2=-t \end{array}}}
  x_{\alpha_1+\alpha_2}^{\hat{\nu}}(m_1)x_{\alpha_1+\alpha_2}^{\hat{\nu}}(m_2)
\end{equation}
and
\begin{equation} \label{R_{12, j}} R_{1,2; j;
    t}=\sum_{{\tiny \begin{array}{c} m \in \Z,\; n \in (1/4) +(1/2) \Z
        \\
        m+n =-t \end{array}}} x_{\alpha_1+\alpha_2}^{\hat{\nu}}(m)
  x_{\alpha_j}^{\hat{\nu}}(n) \; \; \; \mbox{for} \; \; j=1,2.
\end{equation} 
\begin{cor}
  For any $t \in (1/4)\Z$, $R_{j; t}$, $R_{1;2;t}$, $R_{2;1;t}$,
  $R_{1,2; t}$ and $R_{1,2; j; t}$ applied to any vector of $V_L^T$
  give finite sums that equal zero.
\end{cor}

\begin{remark}
  {\em Note that by Lemma \ref{lemma1} we have $R_{2;t}=\pm R_{1;t}$,
  $R_{1;2;t}=\pm R_{1;t}$, $R_{2;1;t}=\pm R_{1;t}$ and
  $R_{1,2;2;t}=\pm R_{1,2;1;t}$}.
\end{remark}

\begin{remark} {\em We can write (\ref{R_j}) as
    follows: \begin{equation} \label{newR_j} R_{j;
        t}=\sum_{{\tiny \begin{array}{c}n_1, n_2 \leq -1/4 \\
            n_1+n_2+1/2=-t \end{array}}} \left
        (x_{\alpha_j}^{\hat{\nu}} \left (n_1+\frac{1}{2} \right )
        x_{\alpha_j}^{\hat{\nu}}(n_2)
        +x_{\alpha_j}^{\hat{\nu}}(n_1)x_{\alpha_j}^{\hat{\nu}} \left
          (n_2+\frac{1}{2} \right ) \right ) +a, \end{equation} where
    $a \in
    U(\bar{\frak{n}}[\hat{\nu}])\bar{\frak{n}}[\hat{\nu}]_{+}$. Indeed,
    by (\ref{bracket2}) in Lemma \ref{brackets} we have
$$
x_{\alpha_1}^{\hat{\nu}}\left ( n_1 + \frac{1}{2} \right )
x_{\alpha_1}^{\hat{\nu}}(n_2)=x_{\alpha_1}^{\hat{\nu}}(n_2)x_{\alpha_1}^{\hat{\nu}}
\left (n_1+\frac{1}{2} \right )+\frac{i}{4}(i^{-4n_1}-(-i)^{-4n_1})
x_{\alpha_1+\alpha_2}^{\hat{\nu}}\left ( n_1+n_2+\frac{1}{2} \right )
$$
and
$$
x_{\alpha_1}^{\hat{\nu}} ( n_1) x_{\alpha_1}^{\hat{\nu}} \left (n_2+
  \frac{1}{2} \right ) = x_{\alpha_1}^{\hat{\nu}} \left (n_2 +
  \frac{1}{2} \right ) x_{\alpha_1}^{\hat{\nu}}
(n_1)-\frac{i}{4}(i^{-4n_1}-(-i)^{-4n_1})
x_{\alpha_1+\alpha_2}^{\hat{\nu}} \left ( n_1+n_2+\frac{1}{2} \right )
$$
for $n_1, n_2 \in \frac{1}{4} \mathbb{Z}$, and thus
$$
x_{\alpha_1}^{\hat{\nu}} \left ( n_1+\frac{1}{2} \right
)x_{\alpha_1}^{\hat{\nu}} (n_2)+ x_{\alpha_1}^{\hat{\nu}} ( n_1)
x_{\alpha_1}^{\hat{\nu}} \left (n_2+ \frac{1}{2} \right )
=x_{\alpha_1}^{\hat{\nu}}(n_2)x_{\alpha_1}^{\hat{\nu}} \left
  (n_1+\frac{1}{2} \right ) + x_{\alpha_1}^{\hat{\nu}} \left (n_2 +
  \frac{1}{2} \right ) x_{\alpha_1}^{\hat{\nu}} (n_1),
$$
which belongs to
$U(\bar{\frak{n}}[\hat{\nu}])\bar{\frak{n}}[\hat{\nu}]_{+}$ for $n_1
\geq 1/4$ or $n_2 \geq 1/4$.  This proves (\ref{newR_j}) for
$j=1$. Similarly one can show (\ref{newR_j}) with $j=2$.}
\end{remark}

We obtain analogous results for (\ref{R_{1,2}}) and (\ref{R_{2,1}}).

We now truncate each of the formal sums above, as follows:
\begin{equation} \label{R_j^0} R_{j;
    t}^0=\sum_{{\tiny \begin{array}{c}n_1, n_2 \in (1/4) + (1/2) \Z \\
        n_1, n_2 \leq -1/4, n_1+n_2+1/2=-t \end{array}}} \left
    (x_{\alpha_j}^{\hat{\nu}} \left (n_1+\frac{1}{2} \right )
    x_{\alpha_j}^{\hat{\nu}}(n_2)
    +x_{\alpha_j}^{\hat{\nu}}(n_1)x_{\alpha_j}^{\hat{\nu}} \left
      (n_2+\frac{1}{2} \right ) \right ) \; \; \; \mbox{for} \; \;
  j=1,2,
\end{equation}
\begin{equation}\label{R_{1,2}^0}
  R_{1;2;t}^0=\sum_{{\tiny \begin{array}{c}n_1, n_2 \in (1/4) + (1/2) \Z \\
        n_1, n_2 \leq -1/4, n_1+n_2+1/2=-t \end{array}}} \left
    (x_{\alpha_1}^{\hat{\nu}} \left (n_1+\frac{1}{2} \right )
    x_{\alpha_2}^{\hat{\nu}}(n_2)
    -x_{\alpha_1}^{\hat{\nu}}(n_1)x_{\alpha_2}^{\hat{\nu}} \left
      (n_2+\frac{1}{2} \right ) \right ),
\end{equation}
\begin{equation}\label{R_{2,1}^0}
  R_{2;1;t}^0=\sum_{{\tiny \begin{array}{c}n_1, n_2 \in (1/4) + (1/2) \Z \\
        n_1, n_2 \leq -1/4,  n_1+n_2+1/2=-t \end{array}}} \left
    (x_{\alpha_2}^{\hat{\nu}} \left (n_1+\frac{1}{2} \right )
    x_{\alpha_1}^{\hat{\nu}}(n_2)
    -x_{\alpha_2}^{\hat{\nu}}(n_1)x_{\alpha_1}^{\hat{\nu}} \left
      (n_2+\frac{1}{2} \right ) \right ),
\end{equation}
\begin{equation} \label{R_{12}^0}
  R_{1,2; t}^0=\sum_{{\tiny \begin{array}{c} m_1, m_2 \in \Z\\
        m_1, m_2 \leq -1, m_1+m_2=-t \end{array}}}
  x_{\alpha_1+\alpha_2}^{\hat{\nu}}(m_1)x_{\alpha_1+\alpha_2}^{\hat{\nu}}(m_2)
\end{equation}
and
\begin{equation} \label{R_{12, j}^0} R_{1,2; j;
    t}^0=\sum_{{\tiny \begin{array}{c} m \in \Z,\; n \in (1/4) +(1/2)
        \Z
        \\
        m\leq -1, n \leq -1/4, m+n =-t \end{array}}}
  x_{\alpha_1+\alpha_2}^{\hat{\nu}}(m) x_{\alpha_j}^{\hat{\nu}}(n) \;
  \; \; \mbox{for} \; \; j=1,2.
\end{equation} 
We will often view (\ref{R_j^0}), (\ref{R_{1,2}^0}),
(\ref{R_{2,1}^0}), (\ref{R_{12}^0}) and (\ref{R_{12, j}^0}) as
elements of $U(\bar{\frak{n}}[\hat{\nu}])$ rather than as
endmorphisms of a $\hat{\frak{g}}[\hat{\nu}]$-module. We have 
$$
R_{2;t}^0=\pm R_{1;t}^0,
$$
$$
R_{2;1;t}^0=-R_{1;2;t}^0=\pm R_{1;t}^0,
$$
$$
R_{1,2;2;t}^0=\pm R_{1,2;1;t}^0.
$$

Set
\begin{equation} \label{ideal}
J
= \sum_{t \geq 1/2} U(\bar{\frak{n}}[\hat{\nu}])R^0_{1;t} 
+ \sum_{t \geq 2}U(\bar{\frak{n}}[\hat{\nu}])
R^0_{1, 2; t} 
+  \sum_{t \geq 5/4}U(\bar{\frak{n}}[\hat{\nu}])
R^0_{1, 2; 1; t} , 
\end{equation} 
the left ideal of $U(\bar{\frak{n}}[\hat{\nu}])$ generated by the
elements (\ref{R_j^0}), (\ref{R_{12}^0}) and (\ref{R_{12, j}^0}) for
$j=1$.  Denote by $I_{\Lambda}$ the left ideal
\begin{equation} \label{ideal-I}
 I_{\Lambda}=J+U(\bar{\frak{n}}[\hat{\nu}])\bar{\frak{n}}[\hat{\nu}]_{+}
\end{equation}
of $U(\bar{\frak{n}}[\hat{\nu}])$. Recall the surjective map
(\ref{f-lambda}). We will prove that the kernel of
$f_{\Lambda}$ is $I_{\Lambda}$.

\begin{remark} \label{homogeneous}
\rm
We have
$$
L^{\hat{\nu}}(0) \; \mbox{Ker} \; f_{\Lambda} \subset \mbox{Ker} \; 
f_{\Lambda}
$$
and
$$
\mbox{wt} \; R^0_{j;t} = \mbox{wt} \; R^0_{1,2; t}=
\mbox{wt} \; R^0_{1,2; j;t}= t 
$$
for $j=1,2$ and $t \in (1/4) \Z$ (recall (\ref{weight_grading_x})).
Note that $R^0_{j;t}$ has charge 2, $R^0_{1,2; t}$ has charge 4 and
$R^0_{1,2;j;t}$ has charge 3 for $j=1,2$, $t \in (1/4)\Z$. The space
$\mbox{Ker} \; f_{\Lambda}$ is also graded by charge. Hence
$\mbox{Ker} \; f_{\Lambda}$ and $I_{\Lambda}$ are graded by both
weight and charge and these two gradings are compatible.
\end{remark}

\section{Certain morphisms}
\setcounter{equation}{0}

Let
\[
\gamma =\frac{1}{2} (\alpha_1+\alpha_2) \; (={\alpha_1}_{(0)}) \in
\frak{h}_{(0)}.
\]
Denote by $ \theta (\cdot)$ the character of the
root lattice $L$,
\begin{equation} \label{character}
\theta : L \longrightarrow \mathbb{C}^{\times},
\end{equation}
such that 
\[
\theta (\alpha_1)= -i, \; \; \;
\theta (\alpha_2) = i.
\]
Define the following map $\tau_{\gamma, \theta}$ on $\bar{\frak{n}}[\hat{\nu}]$:
\begin{eqnarray} \label{tau_map}
  \tau_{\gamma, \theta}: \bar{\frak{n}}[\hat{\nu}] & 
\longrightarrow  & \bar{\frak{n}}[\hat{\nu}] \\
  x_{\alpha}^{\hat{\nu}}(m) & \mapsto & \theta (\alpha)
  x_{\alpha}^{\hat{\nu}}(m + \langle \alpha_{(0)}, \gamma \rangle).
\end{eqnarray}
This is a Lie algebra automorphism (see Lemma \ref{brackets}), and
it extends to an automorphism of $U(\bar{\frak{n}}[\hat{\nu}])$, which
we also denote by $\tau_{\gamma, \theta}$:
\begin{equation} \label{tau_map_ext} \tau_{\gamma, \theta}:
  U(\bar{\frak{n}}[\hat{\nu}]) \longrightarrow
  U(\bar{\frak{n}}[\hat{\nu}]).
\end{equation} 
We have
\begin{multline}
  \tau_{\gamma, \theta}(x_{\alpha_1+\alpha_2}^{\hat{\nu}} (m_1) \cdots 
x_{\alpha_1+\alpha_2}^{\hat{\nu}}(m_r)x_{\alpha_1}^{\hat{\nu}}(n_1) \cdots 
x_{\alpha_1}^{\hat{\nu}}(n_s))\\
  = \theta (\alpha_1)^s x_{\alpha_1+\alpha_2}^{\hat{\nu}} (m_1+1) \cdots 
x_{\alpha_1+\alpha_2}^{\hat{\nu}}(m_r+1)x_{\alpha_1}^{\hat{\nu}}(n_1+1/2) 
\cdots x_{\alpha_1}^{\hat{\nu}}(n_s+1/2)\\
  =(-i)^sx_{\alpha_1+\alpha_2}^{\hat{\nu}} (m_1+1) \cdots
  x_{\alpha_1+\alpha_2}^{\hat{\nu}}(m_r+1)x_{\alpha_1}^{\hat{\nu}}(n_1+1/2)
  \cdots x_{\alpha_1}^{\hat{\nu}}(n_s+1/2)
\end{multline}
for $m_1, \dots, m_r \in \mathbb{Z}$ and $n_1, \dots, n_s \in
\frac{1}{4} \mathbb{Z}$.  Note that
\begin{equation}
\tau_{\gamma, \theta}^{-1}=\tau_{-\gamma, \theta^{-1}}.
\end{equation}

\begin{lem} \label{ideals}
We have 
\begin{equation} \label{tau-ideals} \tau_{\gamma, \theta} \left
    (I_{\Lambda}+
    U(\bar{\frak{n}}[\hat{\nu}])x_{\alpha_1}^{\hat{\nu}}\left
      (-\frac{1}{4} \right ) \right ) =I_{\Lambda}.
\end{equation}
\end{lem}

\begin{proof} 
We shall first show that
\begin{equation} \label{inclusion-one} \tau_{\gamma, \theta} \left
    (I_{\Lambda}+
    U(\bar{\frak{n}}[\hat{\nu}])x_{\alpha_1}^{\hat{\nu}}\left
      (-\frac{1}{4} \right ) \right ) \subset I_{\Lambda}.
\end{equation}
Note that 
\begin{equation} \label{first} \tau_{\gamma, \theta} \left
    (U(\bar{\frak{n}}[\hat{\nu}])\bar{\frak{n}}[\hat{\nu}]_{+}+
    U(\bar{\frak{n}}[\hat{\nu}])x_{\alpha_1}^{\hat{\nu}}\left (
      -\frac{1}{4} \right ) \right ) \subset
  U(\bar{\frak{n}}[\hat{\nu}])\bar{\frak{n}}[\hat{\nu}]_{+}.
\end{equation}
We also have
\[
\tau_{\gamma, \theta} (R_{1;1/2}^0), \; \; \tau_{\gamma,
  \theta}(R_{1;1}^0) \in U(\bar{\frak{n}}[\hat{\nu}]
)\bar{\frak{n}}[\hat{\nu}]_{+},
\]
\[
\tau_{\gamma, \theta}(R_{1;t}^0)=\theta(\alpha_1)^2R_{1; t-1}^0 +a ,
\; \; \; \mbox{where} \; \; a \in
U(\bar{\frak{n}}[\hat{\nu}])\bar{\frak{n}}[\hat{\nu}]_{+}, \; \; t
\geq 3/2,
\]
\[
\tau_{\gamma, \theta}(R_{1,2; 2}^0), \; \tau_{\gamma, \theta}(R_{1,2;
  3}^0) \in U(\bar{\frak{n}}[\hat{\nu}]
)\bar{\frak{n}}[\hat{\nu}]_{+},
\]
\[
\tau_{\gamma, \theta}(R_{1,2;t}^0) =R_{1,2; t-2}^0 +a, \; \; \;
\mbox{where} \; \; a \in
U(\bar{\frak{n}}[\hat{\nu}])\bar{\frak{n}}[\hat{\nu}]_{+}, \; \; t
\geq 4
\]
and
\[
\tau_{\gamma, \theta}(R_{1,2;1;t}^0) \in U(\bar{\frak{n}}[\hat{\nu}]
)\bar{\frak{n}}[\hat{\nu}]_{+} \; \; \; \mbox{for} \; \; 5/4 \leq t
\leq 9/4,
\]
\[
\tau_{\gamma,
  \theta}(R_{1,2;1;t}^0)=\theta(\alpha_1)R_{1,2;1;t-\frac{3}{2}}^0+a,
\; \; \; \mbox{where} \; \; a \in
U(\bar{\frak{n}}[\hat{\nu}])\bar{\frak{n}}[\hat{\nu}]_{+}, \; \; t
\geq 11/4,
\]
and thus 
\begin{equation} \label{second}
\tau_{\gamma, \theta} \left ( J \right ) \subset I_{\Lambda}
\end{equation}
Now (\ref{first}) and (\ref{second}) imply (\ref{inclusion-one}).

By using Lemma \ref{brackets} we get
\[
x_{\alpha_1+\alpha_2}^{\hat{\nu}}(-1) = a R_{1;1}^0+b
x_{\alpha_1}^{\hat{\nu}} \left ( -\frac{5}{4} \right)
x_{\alpha_1}^{\hat{\nu}} \left ( \frac{1}{4} \right ) + c
x_{\alpha_1}^{\hat{\nu}} \left ( - \frac{3}{4} \right )
x_{\alpha_1}^{\hat{\nu}} \left ( - \frac{1}{4} \right ) \in
I_{\Lambda} +
U(\bar{\frak{n}}[\hat{\nu}])x_{\alpha_1}^{\hat{\nu}}\left
  (-\frac{1}{4} \right ).
\]
Then one can see that
\[
U(\bar{\frak{n}}[\hat{\nu}])\bar{\frak{n}}[\hat{\nu}]_{+} \subset
\tau_{\gamma, \theta} \left (I_{\Lambda}+
  U(\bar{\frak{n}}[\hat{\nu}])x_{\alpha_1}^{\hat{\nu}}\left
    (-\frac{1}{4} \right ) \right ).
\]
By the above computations we also have
\[
J \subset \tau_{\gamma, \theta} \left (I_{\Lambda}+
  U(\bar{\frak{n}}[\hat{\nu}])x_{\alpha_1}^{\hat{\nu}}\left
    (-\frac{1}{4} \right ) \right ),
\]
and thus we obtain
\begin{equation}
  I_{\Lambda} \subset  \tau_{\gamma, \theta} \left 
    (I_{\Lambda}+ U(\bar{\frak{n}}[\hat{\nu}])
x_{\alpha_1}^{\hat{\nu}}\left (-\frac{1}{4} \right ) \right ).
\end{equation}
\end{proof}

Define the linear map
\begin{eqnarray} \label{psi_map}
\psi_{\gamma, \theta}: U(\bar{\frak{n}}[\hat{\nu}]) 
& \longrightarrow  & U(\bar{\frak{n}}[\hat{\nu}]) \\
a & \mapsto & \tau_{\gamma, \theta}^{-1} (a) x_{\alpha_1}^{\hat{\nu}} ( - 1/4).
\end{eqnarray}

\begin{lem} \label{sigma_tau_ideals}
We have
\begin{equation} \label{sigma_tau_ideal} \psi_{\gamma, \theta}
  \tau_{\gamma, \theta} (I_{\Lambda}) \subset I_{\Lambda}.
\end{equation} 
\end{lem}

\begin{proof} For any $a \in U(\bar{\frak{n}}[\hat{\nu}])$ we have
  $\psi_{\gamma, \theta} \tau_{\gamma, \theta} (a)=a
  x_{\alpha_1}^{\hat{\nu}}(-1/4)$. First notice that
\begin{equation} \label{one} \psi_{\gamma, \theta} \tau_{\gamma,
    \theta} (U(\bar{\frak{n}}[\hat{\nu}]) \bar{\frak{n}}_{+}) \subset
  U(\bar{\frak{n}}[\hat{\nu}]) \bar{\frak{n}}_{+} \subset I_{\Lambda}.
\end{equation}

Using Lemma \ref{brackets} one can check that
\[
\psi_{\gamma, \theta} \tau_{\gamma, \theta} (R^0_{1;
  t})=x_{\alpha_1}^{\hat{\nu}}(-1/4) R^0_{1; t} + b R^0_{1,2;1;t+1/4}
+ c,
\]
where $b$ is a nonzero constant and $c \in
U(\bar{\frak{n}}[\hat{\nu}]) \bar{\frak{n}}_{+}$, 
\[
\psi_{\gamma, \theta} \tau_{\gamma, \theta}
(R^0_{1,2;t})=x_{\alpha_1}^{\hat{\nu}}(-1/4)R^0_{1,2;t},
\]
and
\[
\psi_{\gamma, \theta} \tau_{\gamma, \theta}
(R^0_{1,2;1;t})=x_{\alpha_1}^{\hat{\nu}}(-1/4)R^0_{1,2;1;t}+ d
R^0_{1,2; t+1/4},
\]
where $d$ is a nonzero constant. Thus
\begin{equation} \label{two} \psi_{\gamma, \theta} \tau_{\gamma,
    \theta} (J) \subset I_{\Lambda}.
\end{equation}

Now (\ref{one}) and (\ref{two}) prove (\ref{sigma_tau_ideal}).
\end{proof}

Consider the linear map
\begin{equation} 
e_{\alpha_1}: V_L^T \longrightarrow V_L^T,
\end{equation}
and its restriction to the principal subspace of $V_L^T$,
\begin{equation} \label{e-alpha}
e_{\alpha_1}: W_L^T \longrightarrow W_L^T.
\end{equation}
Since 
\[
e_{\alpha_1} x_{\alpha}^{\hat{\nu}}(m) = C(\alpha, -\alpha_1)
x_{\alpha}^{\hat{\nu}}(m - \langle \alpha_{(0)}, \alpha_1 \rangle)
e_{\alpha_1}
\]
and
\[
e_{\alpha_1} \cdot 1 = 4 / \sigma (\alpha_1) x_{\alpha_1} ^{\hat{\nu}}
(-1/4) \cdot 1,
\]
we have
\begin{equation}
  e_{\alpha_1}(a \cdot 1) =
  A_{C( \cdot, \cdot )}^{\sigma (\cdot ), \theta(\cdot)} \psi_{\gamma, \theta} (a)
\cdot 1,
\end{equation}
where $a \in U(\bar{\frak{n}}[\hat{\nu}])$ and $A_{C( \cdot, \cdot
  )}^{\sigma (\cdot ), \theta(\cdot)}$ is a nonzero constant depending
on the commutator map (\ref{commutator-definition}), the map
(\ref{sigma}) and the character map (\ref{character}).

We now introduce a twisted version of the $\Delta$-map of
\cite{Li2}.  See more about the (untwisted) $\Delta$-map in the last
section of this paper. Let $\lambda_1$ and $\lambda_2$ be the
fundamental weights of $\frak{sl}(3)$. Then
$$
\lambda_1=\frac{2}{3}\alpha_1+\frac{1}{3}\alpha_2
$$
and
$$
{\lambda_1}_{(0)}=\frac{1}{2}(\alpha_1+\alpha_2), \; \; \; 
{\lambda_1}_{(1)}=\frac{1}{6}(\alpha_1-\alpha_2).
$$

Set
\begin{equation}
\Delta^T(\lambda_1,-x)=i^{2 {\lambda_1}_{(0)}}x^{{\lambda_1}_{(0)}} 
E^+(-\lambda_1,x),
\end{equation}
where $i^{2 {\lambda_1}_{(0)}}=i^{\alpha_1+\alpha_2}$ and
$x^{{\lambda_1}_{(0)}}=x^{(1/2)(\alpha_1+\alpha_2)}$ are operators on the
space $U_T$ and thus on the space $V_L^T$ (see (\ref{x^h}) and
(\ref{action-eta})), and
\[
E^{+}(-\lambda_1, x)= {\rm exp} \left ( \sum_{n \in \frac{1}{4}
    \Z_{+}} \frac{-\lambda_1(n)}{n} x^{-n} \right ) \in \ ({\rm End} \,
V_L^T) [[ x^{-1/4}]]
\]
(see (\ref{exp})), so that
\[
\Delta^T(\lambda_1,-x) \in \ ({\rm End} \,
V_L^T) [[ x^{1/4}, x^{-1/4}]].
\]

Denote by $\Delta^T_c(\lambda_1, -x)$ the constant term of
$\Delta^T(\lambda_1, -x)$. For $m_1, \dots, m_r, n_1, \dots, n_s \in
\frac{1}{4} \Z$, we have
$$ \Delta^T_c(\lambda_1,-x) (x_{\alpha_1+\alpha_2}^{\hat{\nu}}(m_1) 
\cdots x_{\alpha_1+\alpha_2}^{\hat{\nu}}(m_r) x_{\alpha_1}^{\hat{\nu}}(n_1) 
\cdots x_{\alpha_1}^{\hat{\nu}}(n_s)  \cdot 1)$$
$$=\Delta_c^T(\lambda_1, -x) \mbox{Coeff}_{x_1^{-m_1-1} \cdots 
  x_r^{-m_r-1}y_1^{-n_1-1}\cdots y_s^{-n_s-1}}
Y^{\hat{\nu}}(\iota(e_{\alpha_1+\alpha_2}), x_1) \cdots
Y^{\hat{\nu}}(\iota(e_{\alpha_1}), y_s) \cdot 1$$
$$=(-i)^sx_{\alpha_1+\alpha_2}^{\hat{\nu}}(m_1+1) \cdots 
x_{\alpha_1+\alpha_2}^{\hat{\nu}}(m_r+1) x_{\alpha_1}^{\hat{\nu}}(n_1+1/2) \cdots 
x_{\alpha_1}^{\hat{\nu}}(n_s+1/2) \cdot  1,$$
where we have used that
$$E^+(-\lambda_1,x) E^-(-\alpha_1,x_1)=(1-x_1^{1/2}/x^{1/2}) 
E^-(-\alpha_1,x_1) E^+(-\lambda_1,x) $$
and
$$E^+(-\lambda_1,x) E^-(-\alpha_1-\alpha_2,x_1)=(1-x_1/x) 
E^-(-\alpha_1-\alpha_2,x_1) E^+(-\lambda_1,x) .$$

Thus we have the linear map
\begin{eqnarray} \label{CT_map} 
\Delta^T_c(\lambda_1,-x) : W_L^T  & \longrightarrow W_L^T \\
a \cdot 1 & \mapsto & \tau_{\gamma, \theta}(a) \cdot 1.
\end{eqnarray} 

\section{Main results}

\begin{theorem} \label{main-thm}
We have
\begin{equation}
{\rm Ker} \; f_{\Lambda}= I_{\Lambda},
\end{equation}
or equivalently
\begin{equation} \label{main_result}
{\rm Ker} \; \pi_{\Lambda}=I_{\Lambda} \cdot v_{\Lambda}^N.
\end{equation}
\end{theorem}

\noindent {\em Proof:} The proof is analogous to the proof of the
presentation of the principal subspaces in \cite{CalLM3}.  Also recall
Remark 4.1 of \cite{CalLM3}, where we have compared the subtleties of
that proof with those of the corresponding proof in \cite{CalLM1}.

It is easy to see that $I_{\Lambda} \cdot v_{\Lambda}^N \subset {\rm
  Ker} \; \pi_{\Lambda}$. Now assume that ${\rm Ker}\; \pi_{\Lambda}$
is not included in $I_{\Lambda} \cdot v_{\Lambda}^N$. Then there
exists an element $a \in U(\bar{\frak{n}}[\hat{\nu}])$, which we
assume to be homogeneous with respect to the weight and charge
gradings, such that
\begin{equation} \label{assumption} a \cdot v_{\Lambda}^N \in {\rm
    Ker} \; \pi_{\Lambda} \; \; \mbox{but} \; \; a \cdot v_{\Lambda}^N
  \notin I_{\Lambda} \cdot v_{\Lambda}^N.
\end{equation}
Then $a$ is nonzero and nonconstant.  We choose $a$ to be an element
of the smallest possible weight satisfying (\ref{assumption}). Note
that $a$ has positive weight.

We first claim that 
\begin{equation} \label{claim1} a \in I_{\Lambda}+
  U(\bar{\frak{n}}[\hat{\nu}]) x_{\alpha_1}^{\hat{\nu}}(-1/4).
\end{equation}
Assume that (\ref{claim1}) is not true. Then 
\begin{equation} \label{r1} \tau_{\gamma, \theta}(a) \cdot
  v_{\Lambda}^N \notin I_{\Lambda} \cdot v_{\Lambda}^N.
\end{equation}
Indeed, if $\tau_{\gamma, \theta}(a) \cdot v_{\Lambda}^N \in
I_{\Lambda} \cdot v_{\Lambda}^N$, then $\tau_{\gamma, \theta}(a) \in
I_{\Lambda}$, and by Lemma \ref{ideals} we get $a \in I_{\Lambda}+
U(\bar{\frak{n}}[\hat{\nu}]) x_{\alpha_1}^{\hat{\nu}}(-1/4)$, a
contradiction.  Since $a \cdot v_{\Lambda}^N \in {\rm Ker} \;
\pi_{\Lambda}$, then $a \cdot 1=0$ and by applying the map
(\ref{CT_map}) we obtain $\tau_{\gamma, \theta}(a) \cdot 1 =0$, and so
\begin{equation} \label{r2} \tau_{\gamma, \theta}(a) \cdot
  v_{\Lambda}^N \in {\rm Ker} \; \pi_{\Lambda}^N.
\end{equation}
Note that
\begin{equation} \label{r3}
{\rm wt} \; \tau_{\gamma, \theta}(a) < {\rm wt} \; a.
\end{equation}
Now (\ref{r1}), (\ref{r2}) and (\ref{r3}) contradict our choice of the
element $a$ with the property (\ref{assumption}). Thus (\ref{claim1})
holds, so there exist homogeneous elements $b \in I_{\Lambda}$ and $c
\in U(\bar{\frak{n}}[\hat{\nu}])$ such that
\begin{equation} \label{decomposition_a}
a=b+cx_{\alpha_1}^{\hat{\nu}}(-1/4).
\end{equation} 
Then ${\rm wt} \; b= {\rm wt} \; a$ and ${\rm wt} \; c < {\rm wt} \;
a$.

We now claim that
\begin{equation} \label{claim2} c \in I_{\Lambda}+
  U(\bar{\frak{n}}[\hat{\nu}]) x_{\alpha_1}^{\hat{\nu}}(-1/4).
\end{equation}
Assume that (\ref{claim2}) does not hold. Then by Lemma \ref{ideals} we
have
\begin{equation} \label{e1} \tau_{\gamma, \theta}(c) \cdot
  v_{\Lambda}^N \notin I_{\Lambda} \cdot v_{\Lambda}^N.
\end{equation}
On the other hand, 
\begin{equation} \label{e2}
\tau_{\gamma, \theta}(c) \cdot v_{\Lambda}^N \in {\rm Ker} \; \pi_{\Lambda}.
\end{equation}
Indeed, since $a \in {\rm Ker} \; f_{\Lambda}$ and $b \in I_{\Lambda}
\subset {\rm Ker} \; f_{\Lambda}$ we get
\[
0=(a-b) \cdot 1= cx_{\alpha_1}^{\hat{\nu}}(-1/4) \cdot
1=\frac{1}{A_{C( \cdot, \cdot)}^{\sigma(\cdot),
    \theta(\cdot)}}e_{\alpha_1}(\tau_{\gamma, \theta}(c) \cdot 1),
\]
which implies that
\[
\tau_{\gamma, \theta}(c) \cdot 1=0.
\]
Since ${\rm wt} \; \tau_{\gamma, \theta}(c) < {\rm wt} \; a$,
(\ref{e1}) and (\ref{e2}) contradict our choice of the element $a$
satisfying (\ref{assumption}), and therefore (\ref{claim2})
holds. Then
\begin{equation}
  c x_{\alpha_1}^{\hat{\nu}}(-1/4) \in I_{\Lambda} x_{\alpha_1}^{\hat{\nu}}
(-1/4) +I_{\Lambda}.
\end{equation}
Notice that $\psi_{\gamma, \theta} \tau_{\gamma, \theta}
(I_{\Lambda})=I_{\Lambda} x_{\alpha_1}^{\hat{\nu}}(-1/4)$ and thus by
Lemma \ref{sigma_tau_ideals} we obtain that
$I_{\Lambda}x_{\alpha_1}^{\hat{\nu}}(-1/4) \subset I_{\Lambda}$.
Therefore $cx_{\alpha_1}^{\hat{\nu}}(-1/4) \in I_{\Lambda}$, which
gives $a \in I_{\Lambda}$. This shows that our initial assumption is
false, and therefore we have (\ref{main_result}). $\; \; \; \; \; \Box$

Recall the linear maps (\ref{e-alpha}) and (\ref{CT_map}).

\begin{theorem} \label{exact_sequences} We have the following short
  exact sequence of maps:
\begin{equation} \label{ses1} \CD 0 @> > >W_L^T@> e_{\alpha_1} > >
W_L^T @> {\Delta^T_c (\lambda_1, -x)} >> W_L^T @> >> 0 .
\endCD
\end{equation}
\end{theorem}

\noindent {\em Proof:} It is obvious that $e_{\alpha_1}$ is injective,
$\Delta^T_c (\lambda_1, -x)$ is surjective and $$\mbox{Im} \;
e_{\alpha_1} \subset \mbox{Ker} \; \Delta^T_c (\lambda_1, -x).$$
 
Let $v=a \cdot 1 \in \mbox{Ker} \; \Delta^T_c (\lambda_1, -x )$ for $a
\in U(\bar{\frak{n}}[\hat{\nu}])$. Then
\[
0=\Delta^T_c( \lambda_1, -x )(v)=\tau_{\gamma, \theta}(a) \cdot 1.
\]
Therefore $\tau_{\gamma, \theta}(a) \in I_{\Lambda}$ and by Lemma
\ref{ideals} we have
\[
a \in I_{\Lambda}+ U(\bar{\frak{n}}[\hat{\nu}])
x_{\alpha_1}^{\hat{\nu}}(-1/4).
\]
Thus
\begin{equation} \label{formula1} v=a \cdot 1\in \mbox{Ker} \;
  \Delta^T_c (\lambda_1, -x ) \; \; \; \mbox{if and only if} \; \; \; a \in
  I_{\Lambda} + U(\bar{\frak{n}}[\hat{\nu}])
  x_{\alpha_1}^{\hat{\nu}}(-1/4).
\end{equation}

Let $v =a \cdot 1 \in \mbox{Im} \; e_{\alpha_1}$ for $a \in
U(\bar{\frak{n}}[\hat{\nu}])$. Then $v=b
x_{\alpha_1}^{\hat{\nu}}(-1/4) \cdot 1,$ where $b \in
U(\bar{\frak{n}}[\hat{\nu}])$. This implies that
\[
a \in I_{\Lambda} + U(\bar{\frak{n}}[\hat{\nu}])
x_{\alpha_1}^{\hat{\nu}}(-1/4),
\]
and  we get
\begin{equation} \label{formula2} v=a \cdot 1\in \mbox{Im} \;
  e_{\alpha_1} \; \; \; \mbox{if and only if} \; \; \; a \in I_{\Lambda} +
  U(\bar{\frak{n}}[\hat{\nu}]) x_{\alpha_1}^{\hat{\nu}}(-1/4).
\end{equation}

Now (\ref{formula1}) and (\ref{formula2}) give the inclusion
$\mbox{Ker} \; \Delta^T_c (\lambda_1, -x) \subset \mbox{Im} \;
e_{\alpha_1}$. $\; \; \; \; \; \Box$

As we recall from Section 3, the vector space $V_L^T$ has compatible
gradings by weight, given by the action of the Virasoro algebra
operator $L^{\hat{\nu}}(0)$, and by charge, given by the eigenvalues
of the operator $\alpha_1+\alpha_2=(\alpha_1+\alpha_2)(0)$. Restrict
these gradings to $W_L^T$.  In order to make the degrees integers, we shall
now use the weight grading given by $4 L^{\hat{\nu}}(0)$ and the charge grading given
by $\alpha_1+\alpha_2$.  We consider the graded dimension of the
principal subspace $W_L^T$:
\[
\chi(x; q)= {\rm tr}|_{W_L^T}x^{\alpha_1+\alpha_2}q^{4L^{\hat{\nu}}(0)}
\in q^{1/4}\C [[x,q]]
\]
(recall (\ref{weight-1-V_L^T})), where $x$ and $q$ are commuting
formal variables.  As in \cite{CalLM3}, in order to avoid the factor
$q^{1/4}$, we use the following slightly modified graded dimension:
\[
\chi'(x;q)=q^{-1/4} \chi(x;q) \in \C [[x, q]].
\] 

Theorem \ref{exact_sequences} now implies:

\begin{cor}
We have
\begin{equation} \label{euler}
\chi'(x;q)=\chi'(xq^2;q)+xq \chi'(xq^2;q).
\end{equation}
\end{cor}

\noindent {\em Proof:} We denote by ${W_L^T}_{k,l}$ the homogeneous
subspace of $W_L^T$ which consists of elements of charge $k$ and
weight $l$. The exact sequence from Theorem \ref{exact_sequences}
gives
\begin{equation} \label{ses1-again} \CD 0 @> > >{W_L^T}_{k-1,l-2k+1}@>
  e_{\alpha_1} > > {W_L^T}_{k,l} @> {\Delta^T_c (\lambda_1, -x)} >>
  {W_L^T}_{k, l-2k} @> >> 0 ,
\endCD
\end{equation}
and this proves (\ref{euler}). $\; \; \; \; \; \Box$

Thus, as a consequence of the vertex-algebraic theory of principal subspaces in the
case of twisted affine Lie algebras that we have initiated in this paper, we
have obtained a recursion which characterizes the graded
dimension of the principal subspace $W_L^T$.  The solution of this
recursion is given by (cf. \cite{A}):

\begin{cor} We have
$$\chi'(x; q)=\prod_{n \ge 1} (1+xq^{2n-1}).$$
Equivalently, $\chi'(x;q)$ is the two-variable generating function of
the number of partitions of n into m distinct odd parts, which we
denote by $p(O, m, n)$, that is,
\[
\chi'(x; q)= \sum_{m \ge 0} \sum_{n \ge 0} p(O, m, n) x^m q^{n}.
\]
In particular, $\chi'(1;q)$ is the generating function of the number
of partitions of n into distinct odd parts, which we denote by $p(O,
n)$, that is,
\[
\chi'(1;q)=\sum_{n \ge 0} p(O, n)q^n.
\]

\end{cor}

\section{Another proof of the exactness of short sequences for
  principal subspaces in the case of untwisted affine Lie algebras of
  types $A, D, E$}

In this section, we reformulate the result proved in \cite{CalLM3}
giving canonical exact sequences for principal(-like) subspaces of the
level one standard modules for the untwisted affine Lie algebras of types $A, D$ and $E$.
In our reformulation, the role of the intertwining
operators is played by the $\Delta$-map of \cite{Li2}
(not to be confused with the $\Delta_x$ map used earlier).

As in \cite{CalLM3}, let $\goth{g}$ be a finite-dimensional complex simple
Lie algebra of type $A$, $D$ or $E$, of rank $l$.  Let $\goth{h}$ be a
Cartan subalgebra of $\goth{g}$ and let $ \{ \alpha_1, \dots, \alpha_l \}
\subset \goth{h}^{*}$ be a set of simple roots.  Denote
by $\lambda_1, \dots, \lambda_l$ the corresponding fundamental weights
of $\goth{g}$. It is convenient to set $\lambda_0=0$.  Consider the
lattice vertex operator algebra $V_L$ constructed from the root lattice $L$
of $\goth{g}$ (see (\ref{voa}) and (\ref{untwisted-operator})),
and consider the $V_L$-module 
$V_{L+\lambda_i}$ $(=M(1) \otimes \mathbb{C}[L]e^{\lambda_i})$,
$i=0, \dots, l$.  Denote by $Y_{L+\lambda_i}$ the
vertex operator that gives the $V_L$-module structure for
$V_{L+\lambda_i}$. Note that $V_{L+\lambda_i}$ was denoted by
$V_Le^{\lambda_i}$ in \cite{CalLM3}. Denote by $W_L$ the principal
subspace of $V_L$ and by $W_{L+\lambda_i}$ the principal-like
subspaces of $V_{L+\lambda_i}$ for $i=1, \dots,l$, introduced in
\cite{CalLM3}. We refer to \cite{CalLM3} for details.

As in \cite{Li2}, consider the linear map
\begin{equation} 
\Delta(\lambda_i,x)=x^{\lambda_i} E^+(-\lambda_i,-x) \in ({\rm End} \, V_L)[[x, x^{-1}]] ,
\end{equation}
where $E^+( \cdot,x)$ is as in (\ref{exp}), and let 
$$\widetilde{Y}_{V_L}(v,x)=Y_{V_L}(\Delta(\lambda_i,x)v,x) =\sum_{m \in \Z}
\widetilde{v_m}x^{-m-1}.$$ It was proved in \cite{Li2} that
$(V_L,\widetilde{Y}_{V_L})$ is naturally isomorphic to $(V_{L+\lambda_i},
Y_{V_{L+\lambda_i}})$ as  a $V_L$-module, as we discuss below.

In \cite{CalLM3}, we considered an (essentially unique) intertwining
operator
$$\mathcal{Y} (\cdot, x) \in I { V_{L+\lambda_i} \choose 
V_{L+\lambda_i} \ V_L},$$
given by
\begin{equation} \label{intertwining-operator}
\mathcal{Y}(u,x)v=e^{x L(-1)} Y_{V_{L+\lambda_i}}(v,-x)u,
\end{equation}
and its constant term $\mathcal{Y}_c (\cdot, x)$. By projection we
also have the map between principal-like subspaces
$$\mathcal{Y}_c(e^{\lambda_i}, x)  : W_{L} \rightarrow W_{L+\lambda_i},$$
which commutes with the generators of $\bar{\goth{n}}$, where
$\goth{n}$ is the Lie subalgebra of $\goth{g}$ spanned by the root
vectors associated with the positive roots, and $\bar{\goth{n}}$ is as before the
affinization of $\goth{n}$ without central extension.

Recall from \cite{CalLM3}  the linear isomorphism of vector spaces
\[
e_{\lambda_i} : V_L \rightarrow V_{L+\lambda_i}.
\]
It is slightly more convenient to consider (see \cite{Pr}) 
\[
[\lambda_i]=e_{\lambda_i} \circ  c(\cdot,\lambda_i) : 
V_L \rightarrow V_{L+\lambda_i},
\] 
where $c( \cdot,\cdot)$ is the (multiplicative) commutator map as in 
formula (2.17) of \cite{CalLM3}, extended naturally to a linear map
on $V_L$ as in formula (12.2) of \cite{DL2}.  This modified map now satisfies 
$$[\lambda_i] x_{\alpha}(m) =x_{\alpha}(m-\langle \lambda_i,\alpha \rangle)[\lambda_i]$$
for each root $\alpha$.
Moreover, this map is an isomorphism between the $V_L$-modules
$(V_L,\widetilde{Y}_{V_L})$ and
$(V_{L+\lambda_i},Y_{V_{L+\lambda_i}})$:
$$[\lambda_i]^{-1} {Y}_{L+\lambda_i}(v,x) [\lambda_i]=\widetilde{Y}_{V_L}(v,x)$$
(see \cite{Li2}).
Denote by $\Delta_c(\lambda_i, -x)$ the constant term of
$\Delta(\lambda_i, -x)$.  It defines a linear map
\begin{equation}
\Delta_c(\lambda_i, -x): W_L \longrightarrow W_L.
\end{equation} 
Indeed, let us compute the action of $ \Delta_c
(\lambda_i,-x) $ on a ``monomial'' in $W_L$.  For $i=1,\dots, k$, we have
\begin{eqnarray}
  && \Delta_c(\lambda_i,-x)   x_{\alpha_1}(-m_1-1) \cdots  
x_{\alpha_k}(-m_k-1) \cdot {\bf 1} \nonumber \\
  && = \Delta_c (\lambda_i,-x)  {\rm Coeff}_{x_1^{m_1} \cdots x_k^{m_k}} 
Y(e^{\alpha_1},x_1) \cdots  Y(e^{\alpha_k},x_k) \cdot  {\bf 1} \nonumber \\
  &&={\rm CT}_x   {\rm Coeff}_{x_1^{m_1} \cdots x_k^{m_k}}  (-x)(1-x_i/x) 
Y(e^{\alpha_1},x_1) \cdots  Y(e^{\alpha_k},x_k) \cdot {\bf 1}\nonumber \\
  &&=x_{\alpha_1}(-m_1-1) \cdots  x_{\alpha_i}(-m_i) \cdots 
x_{\alpha_k}(-m_k-1) \cdot {\bf 1}, \nonumber
\end{eqnarray}
where $ m_1, \dots, m_k \geq 0$ and  ${\rm CT}_x (\cdots) $ stands
for the constant term. Notice that the shift occurs only for the root 
$\alpha_i$, a consequence of 
$$E^+(-\lambda_i,x) E^-(-\alpha_j,x_j)=
\left(1-\frac{x_j}{x}\right)^{\delta_{ij}}  
E^-(-\alpha_j,x_j) E^+(-\lambda_i,x),$$
where $\delta_{ij}$ is the Kronecker symbol.

We now reformulate Theorem 5.2 in \cite{CalLM3}:
\begin{prop} 
The short exact sequence
\begin{equation}  \CD 0 @> > >W_{L} @> {e_{\alpha_i}} >>
W_L @> {[\lambda_i]^{-1}  \mathcal{Y}_c(e^{\lambda_i}, x)}>> W_{L} @> >> 0 
\endCD
\end{equation}
can be equivalently replaced by 
\begin{equation}  \CD 0 @> > >W_{L} @> {e_{\alpha_i}}> >
W_L @> { \Delta_c (\lambda_i, -x)} >> W_L @> >> 0 .
\endCD
\end{equation}
\end{prop}

\begin{proof} 

For $v \in W_L$, we have
\begin{eqnarray}
&& [\lambda_i]^{-1}  \mathcal{Y}_c (e^{\lambda_i},x)v  = 
[\lambda_i]^{-1}  {\rm CT}_x \mathcal{Y}(e^{\lambda_i},x) v=[\lambda_i]^{-1} 
{\rm CT}_x  \{ e^{L(-1)x} Y_{L+\lambda_i}(v,-x) e_{\lambda_i} \} 
\nonumber \\
&& ={\rm CT}_x e^{L(-1)x} \widetilde{Y}_{V_L} (v,-x){\bf 1}= 
  {\rm CT}_x e^{L(-1)x} Y_{V_L} (\Delta(\lambda_i,-x)v,-x){\bf 1} 
  \nonumber \\
&&={\rm CT}_x e^{L(-1)x} e^{-L(-1)x} \Delta(\lambda_i,-x)v=  
\Delta_c(\lambda_i,-x)v.
\nonumber
\end{eqnarray}
\end{proof}

\vspace{.3in}

\noindent{\small \sc Department of Mathematics, City University of New York, New York City College
  of Technology, New York, NY 11201}\\
{\em E--mail address}: ccalinescu@citytech.cuny.edu\\

\vspace{.1in}

\noindent {\small \sc Department of Mathematics, Rutgers University,
  Piscataway, NJ 08854} \\ {\em E--mail address}:
lepowsky@math.rutgers.edu \\
\vspace{.1in}

\noindent {\small \sc Department of Mathematics and Statistics,
  University at Albany (SUNY), Albany, NY 12222} \\ {\em E--mail
  address}: amilas@math.albany.edu

\end{document}